\documentclass[a4paper,11pt,oneside,dvipsnames,reqno]{amsart} 
\usepackage[foot]{amsaddr}

\usepackage{amsmath,amsfonts,amsthm,amssymb}

\usepackage[UKenglish, shorthands=off]{babel}
\usepackage{csquotes}

\usepackage[T1]{fontenc}
\usepackage{newpxtext,newpxmath}
\usepackage{microtype}
\usepackage{multirow} 
\usepackage{diagbox}
\usepackage{xcolor}
\usepackage{graphicx}
\usepackage{mathtools}

\usepackage{hyperref}
\hypersetup{colorlinks=true, citecolor=PineGreen, linkcolor=RoyalBlue, linktoc=page,urlcolor=RoyalBlue}
\usepackage{cleveref}

\theoremstyle{plain}
\newtheorem{theorem}{Theorem}[section] 
\newtheorem{lemma}[theorem]{Lemma}
\newtheorem{prop}[theorem]{Proposition}
\newtheorem{cor}[theorem]{Corollary}

\newtheorem{conj}{Conjecture}
\newtheorem*{conj*}{Conjecture}

\theoremstyle{remark}
\newtheorem{rem}[theorem]{Remark}

\theoremstyle{definition}

\numberwithin{equation}{section}

\newcommand{\var}{\operatorname{var}}

\newcommand{\A}{\mathbb{A}}

\newcommand{\R}{\mathbb{R}}

\newcommand{\C}{\mathbb{C}}

\newcommand{\Hb}{\mathbb{H}}

\newcommand{\Z}{\mathbb{Z}}

\newcommand{\Q}{\mathbb{Q}}

\newcommand{\Bc}{\mathcal{B}}

\newcommand{\Oc}{\mathcal{O}}

\newcommand{\qf}{\mathfrak{q}}

\newcommand{\new}{\mathrm{new}}

\newcommand{\GL}{\operatorname{GL}}
\newcommand{\PGL}{\operatorname{PGL}}
\newcommand{\SL}{\operatorname{SL}}

\newcommand{\M}{\operatorname{M}}

\newcommand{\sym}{\operatorname{sym}}

\newcommand{\SO}{\operatorname{SO}}

\newcommand{\nr}{\operatorname{nr}}

\renewcommand{\Re}{\operatorname{Re}}
\renewcommand{\Im}{\operatorname{Im}}

\newcommand{\sgn}{\operatorname{sgn}}

\newcommand{\disc}{\operatorname{disc}}

\newcommand{\ad}{\operatorname{ad}}

\newcommand{\vol}{\operatorname{vol}}

\newcommand{\Cl}{\operatorname{Cl}}

\newcommand{\diag}{\operatorname{diag}}

\newcommand{\Stein}{\operatorname{St}}

\newcommand{\abs}[1]{\lvert{#1}\rvert}

\newcommand{\norm}[1]{\left\lVert#1\right\rVert}

\newcommand{\inner}[1]{\langle #1\rangle}

\usepackage[disable]{todonotes}

\usepackage[style = alphabetic, backend = bibtex, abbreviate = true, arxiv = abs, doi = false, isbn = false, url = false, giveninits = true]{biblatex}
\renewbibmacro{in:}{} 
\AtEveryBibitem{%
  \clearlist{language}%
}

\title{Joint equidistribution of newforms}

\author{Asbjørn Christian Nordentoft$^1$}
\address{$^1$University of Copenhagen, Universitetsparken 5, 2100
Copenhagen Ø, Denmark}
\email{nordentoft@math.ku.dk}
\author{Radu Toma$^2$}
\address{$^2$Sorbonne Université, Université Paris Cité, CNRS, IMJ-PRG, F-75005 Paris, France}
\email{toma@imj-prg.fr}

\begin{document}

\begin{abstract}
    Let $Y_1$ be a compact arithmetic hyperbolic surface associated to a maximal quaternion order, let $Y_q$ be a cover associated to an Eichler suborder of prime level $q$, and let $\iota_q$ be embedding of $Y_q$ as the Hecke correspondence into $Y_1 \times Y_1$. Let $\mu_1$ and $\mu_q$ be the invariant probability measures on $Y_1$ and $Y_q$, respectively. If $F$ is a newform on $Y_q$, we conjecture that the pushforward measure $(\iota_q)_\ast(\abs{F}^2 \mu_q)$ converges weakly to the uniform measure $\mu_1 \times \mu_1$, as $q$ tends to infinity. We prove this conjecture with an effective rate of equidistribution, assuming GRH.
\end{abstract}
\maketitle

\tableofcontents

\section{Introduction}
Understanding the distribution of mass of Laplace eigenfunctions on hyperbolic manifolds is a profound subject with roots in quantum chaos (see \cite{Berry02}).
It has proven particularly fruitful to consider arithmetic manifolds and employ number theory to address this problem (see \cite{sarnak-aque}).
This article proposes a novel and deeper study of the interplay between quantum chaos and number theory by investigating an equidistribution problem that combines both quantum chaos and the equidistribution of Hecke points.

Let $B$ be an indefinite quaternion algebra over $\Q$ (in this introduction, the reader may simply consider the split case $B = M_2$). 
Choose a maximal order and, for each integer $q\geq 1$ coprime to the discriminant of $B$, take an Eichler order of level $q$ and define its associated hyperbolic surface $Y_q=\Gamma_q\backslash \Hb$ (in the split case, these are the modular curves $Y_0(q)$ of level $q$).
We endow $Y_1$ and $Y_q$ with invariant probability measures $\mu_1$ and $\mu_q$, respectively.

There are two natural covering maps $\alpha_q,\beta_q: Y_q \rightarrow Y_1$, given by $z\mapsto z$ and $z\mapsto qz$ in the split case.
Together, they define the \emph{Hecke correspondence}
$$\iota_q: Y_q\rightarrow Y_1\times Y_1,$$
by setting $\iota_q=(\alpha_q,\beta_q)$.

If $F$ is a cuspidal Hecke--Maaß newform on $Y_q$, $L^2$-normalized with respect to ~$\mu_q$, then $\mu_F := \abs{F}^2 \mu_q$ defines a new probability measure on~$Y_q$.
It was proposed by Kowalski--Michel--VanderKam \cite{Kowalski2002} that the push-forward measures~$(\alpha_q)_\ast \mu_F$ on~$Y_1$ should converge to the uniform measure~$\mu_1$ as $q$ goes to infinity.
This is an analogue of the original quantum unique ergodicity conjecture of Rudnick and Sarnak \cite{rudnick-sarnak}, which states that $\abs{\psi}^2 \mu_1$ converges to $\mu_1$ as the eigenvalue of a Maaß cuspform $\psi$ on $Y_1$ tends to infinity.
Note that assuming $F$ is a newform is important, although not quite necessary.\footnote{We recall that, at least for $q$ prime, newforms are defined as those automorphic forms orthogonal to old forms of the shape $\alpha_q^\ast(\psi)$ and $\beta_q^\ast(\psi)$ for $\psi$ defined on $Y_1$.
Clearly, requiring that $F$ be a newform in the conjecture of \cite{Kowalski2002} is wise, since the measures attached to old forms $\alpha_q^\ast(\psi)$ for some cuspform $\psi$ on $Y_1$ are always equal to $\mu_{\psi}$, which is independent of $q$.
However, the newform assumption is not quite necessary, and it is an amusing observation that the measures attached to old forms $\beta_q^\ast(\psi)$ \emph{do} converge to the uniform measure.
Indeed, this reduces to the equidistribution of Hecke points, as can be seen in the definition of the Hecke operator in \eqref{eq:hecke-op} below.}

We propose the following stronger conjecture, which combines the two maps $\alpha_q$ and $\beta_q$ and takes the equidistribution of the Hecke correspondence into account.
\begin{conj}\label{conj:mixingQUE}
    Let $(F_q)$ be a sequence of Hecke--Maaß newforms, where $F_q$ is defined on $Y_q$.
    Then, as $q$ goes to infinity, the measures $(\iota_q)_\ast \mu_{F_q}$ converge to the uniform measure $\mu_1 \otimes \mu_1$ on $Y_1 \times Y_1$.
\end{conj}
A similar conjecture can be made for holomorphic modular forms.
Notice that, since $\alpha_q$ is equal to the composition of $\iota_q$ and the projection to the first factor, this conjecture implies the one in \cite{Kowalski2002}.

In \cite{nelson}, Nelson proves the Kowalski--Michel--VanderKam conjecture for holomorphic modular forms, generalizing the work of Holowinsky--Soundararajan \cite{holo-sound} and Watson \cite{watson}.
Nelson's work also shows that the case of Maaß forms follows from conjectured subconvexity results for certain $L$-functions, or the Generalized Lindelöf Hypothesis (GLH), more generally.

The main result of the present paper is the following.
\begin{theorem}\label{thm:main-intro}
For $B$ a division algebra, along a sequence of prime levels $q$, Conjecture~\ref{conj:mixingQUE} holds under the Generalized Riemann Hypothesis (GRH) with an effective rate $(\log q)^{-1/4+\varepsilon}$ for any $\varepsilon>0$. 
\end{theorem}

\subsection{Comparison with the mixing conjecture}
Conjecture~\ref{conj:mixingQUE} and Theorem~\ref{thm:main-intro} share many features with the Mixing Conjecture of Michel--Venkatesh \cite{michel-venkatesh} and the currently available partial results towards it.
To state it, let $D < 0$ be a fundamental discriminant and let $\Cl(D)$ be the class group of $\Q(\sqrt D)$.
Take a class $\sigma \in \Cl(D)$ and denote by $q$ be the minimal norm of an integral ideal in $\sigma$.
Now consider a sequence of such data $(D_i, \sigma_i, q_i)$ and let $\mathcal{H}_{D_i}$ be the set of Heegner points in the modular curve $Y_0(1)$ attached to $\Cl(D_i)$.
Michel and Venkatesh conjecture that, as $q_i$ tends to infinity (and $D_i \to \infty$ implicitly, as discussed below), the sets
\begin{displaymath}
    \{(z, \sigma_i z) \mid z \in \mathcal{H}_{D_i}\} \subset Y_0(1) \times Y_0(1)
\end{displaymath}
equidistribute with respect to the uniform product measure $\mu_1 \otimes \mu_1$.
Note that, when projecting to one of the factors, equidistribution was already shown by Duke in \cite{duke} using subconvexity of $L$-functions (in the guise of bounds for Fourier coefficients of half-integral weight forms).

We recall Minkowski's theorem (see \cite[Sec.~6.5]{michel-venkatesh}), which implies that $q = O(\sqrt{\abs{D}})$.
In particular, if $q \to \infty$, then $D \to \infty$, the latter being a condition that is commonly found in the statement of the Mixing Conjecture.
We also observe that the set $\{(z, \sigma_i z) \mid z \in \mathcal{H}_{D_i}\}$ lies in the Hecke correspondence $\iota_q(Y_0(q))$, as explained in Section 2.2 of \cite{blomer2022mixingconjecturegrh}.

Thus, the Mixing Conjecture is about point masses inside a Hecke correspondence and their equidistribution inside the product space, given that the projections are already known to equidistribute.
This is similar to Conjecture~\ref{conj:mixingQUE}, where we replace the points by the mass of an automorphic form.
However, we note that the required subconvexity to prove the equidistribution in the projections, i.e. the conjecture of Kowalski--Michel--VanderKam, is not yet known in full generality.

To deepen the analogy, we suggest comparing the quantity $D/q^2$ to the Laplace eigenvalue $\lambda_F$.
Minkowski's theorem shows that $D/q^2 \gg 1$ and the uniform spectral gap or Selberg's conjecture imply that $\lambda_F \gg 1$.
Moreover, if $D/q^2 \gg q^\eta$ and $\lambda_F \gg q^\eta$ for some $\eta > 0$, we expect both conjectures to be easier.
This is the idea of equidistribution \emph{in stages} explored by Ellenberg--Michel--Venkatesh \cite{EMV} for the Mixing Conjecture (see  \cite[Sec.~2.5]{blomer2022mixingconjecturegrh} for a nice discussion).
Namely, if measures on subsets $Y_i$ of a space $X$  converge fast enough to the uniform measure on $Y_i$ and, in turn, the uniform measures on $Y_i$ converge fast enough to the uniform measure on $X$, then one may piece these facts together and obtain full equidistribution.

In both settings, the Hecke correspondences equidistribute fast by the spectral gap of Hecke operators.
In the Mixing Conjecture, we recall that $\mathcal{H}_D$ contains approximately $\sqrt{D}$ points and the $q$-th Hecke correspondence has volume approximately $q$.
We thus need $\sqrt{D}$ much larger than $q$ to have good equidistribution in the correspondence.
Similarly, for a version of quantum unique ergodicity to manifest itself on the correspondence, we would need a large eigenvalue $\lambda_F$.
In Section \ref{sec:equid-stages}, we show that GLH is enough to prove our conjecture under the assumption $\lambda_F \gg q^\eta$.
The difficult range $\lambda_F \asymp 1$ then requires the stronger input of GRH.

In terms of results towards the two conjectures, the present paper plays a similar role as the work of Blomer--Brumley \cite{Blomer2024} and Blomer--Brumley--Khayutin \cite{blomer2022mixingconjecturegrh}, where the Mixing Conjecture and related results are proved under GRH.
We point out that, using spectral methods, Conjecture~\ref{conj:mixingQUE} involves $L$-functions of higher degree, namely
\begin{displaymath}
    L(1/2, \sym^2 F \times \varphi) L(1/2, \varphi),
\end{displaymath}
where $F$ and $\varphi$ are newforms, whilst Blomer--Brumley--Khayutin deal with
\begin{displaymath}
    L(1/2, \varphi \times \chi_D) L(1/2, \varphi),
\end{displaymath}
where $\chi_D$ is a quadratic character\footnote{We note that, for the difficult range of parameters $q\asymp D^{1/2}$, both \cite{Blomer2024} and \cite{blomer2022mixingconjecturegrh} consider a spectral expansion, not over the Hecke correspondence, but over the torus of discriminant $-D$. This results in the relevant family being given by twists by class group characters $\chi$, of the shape $L(1/2, f_1\otimes \theta_\chi)L(1/2, f_2\otimes \theta_\chi)$, where $f_1,f_2$ are Hecke--Maa{\ss} forms on $Y_1$ and $\theta_\chi$ denotes the theta series associated to $\chi$.}.

\subsection{Proof sketch} 
Our approach is spectral in nature, relying on the Watson--Ichino formula, Proposition~\ref{prop:watson-ichino} below, to relate the problem to $L$-functions. 
We recall that, in the context of Theorem~\ref{thm:main-intro}, the spaces $Y_1$ and $Y_q$ are compact.
Thus, to show equidistribution by Weyl's criterion, Lemma~\ref{lem:Weyl} below, it suffices to show for any $f_1,f_2$ Hecke--Maa{\ss} forms on $Y_1$, not both constant, that
\begin{displaymath}
    \langle |F|^2, f_1 \cdot (l_q f_2)\rangle_q\rightarrow 0,\quad \text{as } q\rightarrow \infty,
\end{displaymath}
where $\langle\cdot,\cdot\rangle_q$ is the inner product on $Y_q$ normalized so that $\langle 1,1 \rangle_q=1$ and $(l_q f_2)(z):= (\beta_q)_\ast f_2(z)$ (for the split algebra, this is $z \mapsto f_2(qz)$). 
If either one of $f_1,f_2$ is constant, then we are reduced to the setting of the Kowalski--Michel--VanderKam conjecture and we have by the Watson--Ichino formula that
\begin{displaymath}
  |\langle |F|^2, f_1  f_2)\rangle_q|^2 \approx q^{-1} (|t_F|+1)^{-1/2} L(1/2,F\otimes F\otimes f_i)
\end{displaymath}
which is $o(1)$ given any subconvexity bound in $q$ and convexity exponent in $t_F$.
This is the content of Proposition~\ref{prop:que-level-glh}.

If both $f_1,f_2$ are non-constant, we get by spectrally expanding over a Hecke eigenbasis $\Bc_0(Y_q) \cup \{1\}$ of $Y_q$ that 
\begin{displaymath}
  \langle |F|^2, f_1 \cdot (l_q f_2)\rangle_q =\langle 1, f_1 \cdot (l_q f_2)\rangle_q+\sum_{ \varphi \in \Bc_0(Y_q) }   \langle |F|^2, \varphi \rangle_q  \langle \varphi, f_1 \cdot (l_q f_2)\rangle_q.
\end{displaymath}
Next, we apply the Watson--Ichino formula and the description of the Hecke operator in terms of the Hecke correspondence, as given by Lemma~\ref{lemma:hecke-corr-op}, obtaining that the expression above is
\begin{equation} \label{eq:intro-spec-expansion}
    \approx \delta_{f_1=f_2}\tilde{\lambda}_{f_2}(q)+q^{-1} (t_F+1)^{-1/4}\sum_{ \varphi \in \Bc_0(Y_q) , t_\varphi\leq 100 } \pm \mathcal{L}(\varphi)^{1/2} ,
\end{equation}
where\footnote{Here we are ignoring all factors including $L$-functions at $s=1$.}
$$\mathcal{L}(\varphi)=L(1/2,F\otimes F\otimes \varphi)L(1/2,f_1\otimes f_2\otimes \varphi),$$
$\tilde{\lambda}_{f_2}(q)$ denotes the geometrically normalized Hecke eigenvalue (so that Ramanujan's conjecture is $|\tilde{\lambda}_{f_2}(q)|\ll q^{-1/2} $) and $\pm$ denotes a sign. 
Here the truncation $t_\varphi\leq 100$ follows from the decay of the archimedean component of $L(1/2, f_1\otimes f_2\otimes \varphi)$. 
If we insert GLH, the Weyl law and bounds towards Ramanujan, then we may deduce that \eqref{eq:intro-spec-expansion} is
\begin{displaymath}
    \ll q^{-25/64}+ (t_F+1)^{-1/4+\varepsilon} q^\varepsilon,
\end{displaymath}
for any $\varepsilon > 0$.
Thus, we obtain the equidistribution result for $t_F\geq q^{\eta}$ for some fixed $\eta>0$, as $q\rightarrow \infty$. 
This is explored in more detail in Section \ref{sec:equid-stages} and matches exactly the idea of equidistribution \emph{in stages}, alluded to above.

We conclude that the hardest range is when $t_F \asymp 1$, in which case, under GLH, we  ``lose by an $\varepsilon$''. 
We again take absolute values\footnote{There is no known way to control the cancellation, which is most likely there, coming from the signs $\pm$.} in \eqref{eq:intro-spec-expansion} and reduce to a fractional moment problem, which requires us to access finer statistical information about $L$-functions at the central point $s=1/2$. 
We use the crucial fact that, on \emph{average} over $\varphi\in \Bc_0(Y_q)$, the two $L$-functions $L(1/2,F\otimes F\otimes \varphi)$ and $L(1/2,f_1\otimes f_2\otimes \varphi)$ are a little less than 1, of size $(\log q)^{-\delta}$ for some $\delta>0$. 
Thus, if the two families of $L$-functions are sufficiently independent, we would expect the average of the product to decay.

More precisely, we use that $\log L(1/2,F\otimes F\otimes \varphi)$ and $\log L(1/2,f_1\otimes f_2\otimes \varphi)$ behave like a (non-degenerate) two-dimensional Gau{\ss}ian distribution with mean and variance on the scale of $\log \log q$. 
We refer to Section \ref{sec:heuristics} for details on these heuristics and a calculation showing that they imply decay of the Weyl sums. 
Making these heuristics precise relies on a technique of Soundararajan \cite{Sound09}, which allows one to approximate $L$-functions at the central point by, essentially, a very short Euler product, \emph{assuming GRH}. 
Then, using the Kuznetsov trace formula, one can calculate high moments of short Dirichlet polynomials, which are consistent with the Gau{\ss}ian heuristics. 
Putting all of this together yields the bound 
\begin{displaymath}
    \sum_{\varphi\in \Bc_{0}(Y_q): t_\varphi\leq 100} \mathcal{L}(\varphi)^{1/2}\ll\begin{cases}
            q (\log q)^{-3/8+\varepsilon},& f_1\neq f_2;\\
            q (\log q)^{-1/4+\varepsilon},& f_1= f_2,
        \end{cases}
\end{displaymath}
as shown in our main technical result, Theorem~\ref{thm:mainestimate}.

The first instance of this technique, assuming GRH, in the context of equidistribution seems to be the work of Lester--Radziwiłł \cite{lester-radz} on QUE for half-integral weight forms. 
Soon after, Blomer--Brumley \cite{Blomer2024} applied it in the context of simultaneous equidistribution for toroidal orbits (simultaneous Duke's theorems), which was then used by Blomer--Brumley--Khayutin \cite{blomer2022mixingconjecturegrh} in the setting of the Mixing Conjecture\footnote{Note that in \cite{blomer2022mixingconjecturegrh} the Generalized Ramanujan Conjecture is also assumed but for a different part of the argument (the sieving part).}.
We point out also the recent work of J\"{a}\"{a}saari--Lester--Saha \cite{JaasaariLesterSaha25} on QUE for lifts on $\mathrm{GSp}(4)$ and of Huang \cite{bingrong} on joint equidistribution of eigenforms in the spectral aspect, assuming also the Generalized Ramanujan Conjecture (GRC). 

\begin{rem}\label{rem:eisenstein}
   Note that, due to the fact that $Y_q$ is compact, there is no continuous spectrum in spectral expansion. We are not currently able to resolve the fractional moment problem when $f_1,f_2$ are Eisenstein series. 
   This is analogous to the situation in \cite{Blomer2024}. 
   We can, however, bound the \emph{cuspidal} Weyl sums in the non-compact case $Y_0(q)$. 
\end{rem}

\subsubsection{New aspects of the proof}
The blueprint for bounding fractional moments under GRH was laid out in \cite{lester-radz}. 
In \emph{loc.\ cit.}\ the authors introduced a certain splitting into different ranges, according to the size of $\log \mathcal{L}(\varphi)$, and then applied different high moments estimates in each of the cases to bound the fractional moment. 
Applying this blueprint implies that, if one assumes GRH \emph{and} GRC, then one can recover the bound one obtains from the heuristics in Section \ref{sec:heuristics} (up to a factor $(\log q)^\varepsilon$). 
Here, we note that executing the heuristics involves analyzing the holomorphicity of the $L$-functions of various functorial lifts (e.g.\ $L(s,\sym^4 F)$). 

We now list some of the new features appearing in our treatment, comparing to \cite{lester-radz} and \cite{Blomer2024}. 
In particular, a lot of work goes into controlling various factors which are known to be small under both GRH and GRC.
\begin{itemize}
    \item First of all, we introduce a slight simplification of the splitting of ranges: by \emph{first} doing the change of variable by the mean $\mu_q$ and \emph{then} ruling out very large values, we avoid using any second moment estimates as in \cite[Lemma~5.6]{lester-radz} or \cite[Sec.~8]{Blomer2024}. See Table \ref{table:ranges} for details.
    
    \item The $L$-functions that show up are of degree up to $12$ and this pushes our knowledge of functoriality to the limit. 
    In particular, the automorphy of $\sym^4$ due to Kim \cite{Kim03} and the fact that the bound towards Ramanujan due to Kim--Sarnak \cite[Appendix 2]{Kim03} satisfies $\frac{7}{64}<1/8$ are crucial for our arguments. 
    These facts are gathered in Section \ref{sec:automorphicstuff} and appear in many places throughout the argument.
    
    \item In the case $f_1 = f_2$, the two $L$-functions we consider are actually \emph{not} independent on the $\log \log q$-scale, as opposed to the ones appearing in \cite{Blomer2024} and \cite{lester-radz}. 
    However, the dependence is fortunately compensated for by the mean of $\log L(1/2,f_1\otimes f_2\otimes \varphi)$ being a factor of $2$ more negative when $f_1=f_2$.
    
    \item \todo{One could say: 'A lastly, a very technical point.'}The method of Soundararajan yields an upper bound for $\log \mathcal{L}(\varphi)$ corresponding to taking the logarithm of a short truncation ($p<x$ for some $x\geq 2$) of the relevant Euler product, but with weights of the slightly awkward shape $p^{-1/\log x} \frac{\log x/p}{\log x}$. 
    Removing these weights without GRC requires some additional work.
    For comparison, in \cite{Blomer2024}, a certain positivity argument is used that is not available in our case. 
    In fact, in the case of $L(1/2,F\otimes F\otimes \varphi)$, the corresponding term is \emph{negative}. 
    To circumvent this, we do a change of variables by $(1-\varepsilon)\mu_q$ instead of the mean $\mu_q$ itself and use the negativity to our advantage. To make this work we are required to truncate at $x$ of larger size  compared to \cite[Cor.~7]{Blomer2024}, see eq.\ (\ref{eq:x}) below. 
    Ultimately, since the exponents in Lemmas \ref{lem:bulkrangezx} and \ref{lem:largerange1/2} are allowed to depend on $\varepsilon$, this is admissible. We refer to Remark \ref{rem:largerxGRHbound} for more details on this technical point.  
\end{itemize}

\begin{rem}\label{rem:generalq}
 The condition that $q$ is prime can be slightly weakened (e.g. a product of a bounded number of primes of roughly the same size).
 However, in the case where $q$ is, say, the product of the first $n$ primes, some new input is needed to deal with the old spectrum.
\end{rem}

\section{The Hecke correspondence} \label{sec:hecke}
\subsection{Quaternion algebras and arithmetic hyperbolic surfaces}
Let $B$ be an indefinite quaternion algebra over $\Q$ with reduced discriminant $\disc(B)=D_B$ and let $\Oc \subset B$ be a maximal order.
For a prime $q$ such that $(q, \disc(B)) = 1$, let $a_q \in \Oc$ satisfy $\nr(a_q) = q$ (for existence see Eichler's theorem on norms \cite[Cor.~28.6.1]{voight} and the fact that all maximal orders are conjugate \cite[Cor.~28.5.6]{voight}), and define
\begin{displaymath}
    \Oc_0(q) = a_q^{-1}\Oc a_q \cap \Oc.
\end{displaymath}
Thus, $\Oc_0(q)$ is an Eichler order (see \cite[Sec.~23.4]{voight}). 

Locally, at a prime $p \neq q$, the order $\Oc_0(q)_p$ is maximal, equal to $\Oc_p$.
At the prime $q$, the order $\Oc$ is isomorphic to $\M_2(\Z_q)$ and, under this isomorphism, $\Oc_0(q)_q$ is conjugate to the order
\begin{displaymath}
    \begin{pmatrix}
        \Z_q & \Z_q \\
        q \Z_q & \Z_q
    \end{pmatrix}.
\end{displaymath}

Fix an embedding of $B$ into $\M_2(\R)$.
Let $\Gamma_q := \Oc_0(q)^1$ be the group of norm 1 units of $\Oc_0(q)$ as a discrete subgroup of $G = \SL_2(\R)$ and define $Y_q = \Gamma_q \backslash \mathbb{H}$, where $\mathbb{H} \cong \SL_2(\R)/\SO(2)$ is the hyperbolic plane.
Analogously, define $\Gamma_1 := \Oc^1$ and $Y_1 = \Gamma_1 \backslash \mathbb{H}$ using the maximal order.

For example, if $B = M_2(\Q)$, $\Oc = M_2(\Z)$, and $a_q = \diag(q,1)$, then $\Gamma_q$ is the standard Hecke congruence subgroup $\Gamma_0(q)$.
Otherwise, if $B$ is a division algebra, then $Y_1$ and $Y_q$ are compact hyperbolic surfaces (see \cite[Sec. 38.4]{voight}).

Equip $Y_q$ with the $G$-invariant probability measure~$\mu_q$ and let~$\inner{\cdot, \cdot}_q$ be the corresponding inner product on $L^2(Y_q)$ and $\norm{\cdot}$ the corresponding norm.\footnote{We remove the subscript $q$ from norms to avoid confusion with $L_q$-norms. It is clear from context. See Remark \ref{rem:norms} for additional comments.}
We also consider $Y_1 \times Y_1$ as a space with a~$G \times G$-action and equip it with the invariant probability measure $\mu_1 \otimes \mu_1$.
These invariant measures with full support are also often called uniform measures.

\subsection{Hecke correspondences and operators}
Let
\begin{displaymath}
    \iota_q : Y_q \to Y_1 \times Y_1
\end{displaymath}
be the map that sends a matrix~$g \in G$ to the pair $(g, a_q g)$.
It is well-defined, since $a_q \Gamma_q a_q^{-1} \subset \Gamma_1$, and injective, since $\Gamma_q = a_q^{-1} \Gamma_1 a_q \cap \Gamma_1$.

The image of $\iota_q$ is called the $q$-th \emph{Hecke correspondence}.
This name is also used to describe the system of maps
\begin{displaymath}
    \alpha_q, \beta_q : Y_q \to Y_1, \quad \alpha_q(g) = g, \quad \beta_q(g) = a_q \cdot g.
\end{displaymath}
They are, respectively, the composition of $\iota_q$ with the projections $\pi_1, \pi_2$ onto the first and second factor of $Y_1 \times Y_1$.
When pulling back functions through $\alpha_q$ and $\beta_q$, we often simplify notation by writing $\alpha_q^\ast f(z) = f(z)$ and $\beta_q^\ast f(z) = f(a_q z) = l_q f(z)$, by viewing functions on $Y_1$ as periodic functions on $\mathbb{H}$.

By pulling back functions through $\beta_q$ and pushing forward through $\alpha_q$, we obtain the Hecke operator $T_q$, that is,
\begin{equation} \label{eq:hecke-op}
    \sum_{w \in \alpha_q^{-1}(z)} f(\beta_q(w)) = \sum_{\gamma \in \Gamma_1 \backslash \Gamma_1 a_q \Gamma_1} f(\gamma z)
\end{equation}
for functions $f$ defined on $Y_1$.
In this normalization, the Ramanujan conjecture says that $T_q$-eigenvalues of cuspforms are bounded by $2\sqrt{q}$.
It is well-known that the degree of $T_q$ is equal to the index $[\Gamma_1 : \Gamma_q]$ and we write $\tilde{T}_q := [\Gamma_1 : \Gamma_q]^{-1} T_q$.
Finally, define $\theta \in [-1/2, 0]$ to be a bound towards the Ramanujan conjecture, that is, $\tilde{T}_q$ acts on full level cuspidal eigenfunctions with eigenvalues bounded by $q^{\theta}$ in absolute value.
The trivial bound is $\theta = 0$ and the best currently available bound is $\theta = -25/64$, due to Kim--Sarnak (see \eqref{eq:KimSarnak}).
Note also that the eigenvalues of $T_q$ on functions on $Y_1$ are real by self-adjointness.

It is useful to state the previous observation in the following form, which follows immediately after unpacking normalizations.
\begin{lemma} \label{lemma:hecke-corr-op}
    If $f_1$ and $f_2$ are functions on $Y_1$, then
    \begin{displaymath}
        (\iota_q)_\ast \mu_q(f_1 \otimes f_2) = \inner{f_1, l_q\bar{f}_2}_q = \inner{f_1, \tilde{T}_q \bar{f}_2}_1.
    \end{displaymath}
\end{lemma}

\begin{rem} \label{rem:norms}
    To avoid any confusions regarding normalizations, we collect here a few observations.
    Let $f: Y_1 \to \C$ and assume that $\inner{f, f}_1 = 1$.
    Then we have $\inner{\alpha_q^\ast f, \alpha_q^\ast f}_q = 1$ and we often write $f$ instead of $\alpha_q^\ast f$ by thinking of it as a function on $\mathbb{H}$.
    Therefore, there is no confusion in simply writing $\norm{f}$.
\end{rem}

\subsection{Involutions} \label{sec:invol}
We also note the existence of the involution $W$ on $Y_1 \times Y_1$ given by switching components, i.e. $W(g, h) = (h, g)$.    
Observe that, on the Hecke correspondence $Y_q$, it induces the Fricke involution $W_q$.
In the case of $a_q = \diag(q,1)$, $\Gamma_q = \Gamma_0(q)$, this is defined as
\begin{displaymath}
    W_q(g) = \begin{pmatrix}
        0 & -1 \\
        q & 0
    \end{pmatrix} g = a_q w g,
\end{displaymath}
where $w$ is the nontrivial Weyl group element.
For general quaternion algebras, we may follow \cite{ogg} and argue locally to essentially reduce to the operator above.

We recall that there is a prime ideal $\qf \subset \Oc_0(q)$ such that $\qf^2 = q \Oc_0(q)$ (see \cite[Sec.~23.3]{voight}).
Eichler's theorem (see \cite[Thm.~28.5.5]{voight}) implies that $\qf = w_q \Oc_0(q) = \Oc_0(q) w_q$ is principal, where $w_q \in \Oc_0(q)$ and $\nr(w_q) = q$.
Thus $w_q$ normalizes $\Gamma_q$ and we define 
\begin{displaymath}
    W_q(g) = w_q g.
\end{displaymath}

To relate $a_q$ and $w_q$, it is useful to work locally.
Fix an isomorphism $\Oc_q \cong \M_2(\Z_q)$.
Since $\nr(a_q) = q$, a version of Smith's normal form implies that there exist $\xi_1, \xi_2 \in \SL_2(\Z_q)$ such that
\begin{displaymath}
    a_q = \xi_1 \begin{pmatrix}
        q & \\ & 1
    \end{pmatrix} \xi_2,
\end{displaymath}
and
\begin{displaymath}
    \xi_2 \Oc_0(q)_q \xi_2^{-1} = \begin{pmatrix}
        \Z_q & \Z_q \\
        q \Z_q & \Z_q
    \end{pmatrix}.
\end{displaymath}
The prime ideal $\qf_q \subset \Oc_0(q)_q$ is then given by
\begin{displaymath}
    \xi_2^{-1} \begin{pmatrix}
        q\Z_q & \Z_q \\
        q \Z_q & q\Z_q
    \end{pmatrix} \xi_2 = \xi_2^{-1} 
    \begin{pmatrix}
        & 1\\ q&
    \end{pmatrix}
    \Oc_0(q)_q \xi_2 = \xi_2^{-1} 
    \Oc_0(q)_q 
        \begin{pmatrix}
        & 1\\ q&
    \end{pmatrix}\xi_2.
\end{displaymath}
Thus, the localization of $w_q$ at $q$ has the form
\begin{displaymath}
    \xi_2^{-1} \begin{pmatrix}
        & 1\\ q &
    \end{pmatrix} \gamma \xi_2 =  \xi_2^{-1} \gamma' \begin{pmatrix}
        & 1\\ q &
    \end{pmatrix}\xi_2,
\end{displaymath}
for some $\gamma, \gamma' \in \Oc_0(q)_q^\times$.
It is now easy to check, by direct computation using the local descriptions above, that $a_q w_q \in q \cdot \Oc^1$ and $w_q = \xi \cdot a_q$ for some $\xi \in \Oc^1$.
This directly implies the identity 
\begin{equation} \label{eq:fricke-invariance}
    \iota_q \circ W_q = W \circ \iota_q.
\end{equation}

\begin{rem} \label{rem:fricke-w-invariance}
The Fricke involution $W_q$ acts on a newform $\phi_q$ by multiplication by $\pm 1$, so it leaves the measure $\mu_{\phi_q}$ invariant.
The relation \eqref{eq:fricke-invariance} implies that $(\iota_q)_\ast \mu_{\phi_q}$ is invariant under the involution $W$ switching components.
\end{rem} 

\section{$L$-functions and triple products}
The main goal of this section is to recall definitions of $L$-functions and state the triple product formula. 
We also gather important automorphic facts and useful computations with Hecke eigenvalues.

\subsection{$L$-functions} \label{sec:l-functions}
We say that a meromorphic function $L(s,\pi)$ with at most poles at $s=0,1$ is an \emph{$L$-function of degree $d$} if we have the following data:
\begin{enumerate}
    \item There is a Dirichlet series with an Euler product of degree $d\geq 1$ 
    \begin{equation}
        \label{eq:Eulerprod}L(s,\pi)=\sum_{n\geq 1}\frac{\lambda_\pi(n)}{n^s}=\prod_p (1-\alpha_\pi(p, 1)p^{-s})^{-1} \cdots (1-\alpha_\pi(p, d)p^{-s})^{-1},
    \end{equation}
    converging absolutely for $\Re s>1$. We refer to $\alpha_\pi(p,i)$ as the \emph{Satake parameters at $p$}.  
    \item A gamma factor
    $$ L_\infty(s,\pi)=\pi^{-ds/2}\prod_{j=1}^d \Gamma\left(\frac{s+s_\pi(j)}{2}\right). $$
    We refer to $s_\pi(j)$ as the \emph{Satake parameters at $\infty$}.
    \item An integer $q(\pi)\geq 1$, called the conductor of $L(s)$ such that $\alpha_\pi(p, i)\neq 0$ for $p\nmid  q(\pi)$ and $1 \leq i \leq d$, and a functional equation
    \begin{equation}
        \Lambda(s,\pi):=q(\pi)^s L_\infty(s,\pi)L(s,\pi)=\epsilon(\pi) \Lambda(1-s,\overline{\pi}), 
    \end{equation}
    where  $\epsilon(\pi)$ is of absolute value 1 and referred to as the \emph{root number} and $\overline{\pi}$ is the \emph{dual of $\pi$} satisfying $\lambda_{\overline{\pi}}(n)=\overline{\lambda_\pi(n)}$. 
\end{enumerate}
 We define the \emph{analytic conductor of $\pi$} as 
\begin{equation}
    \mathbf{c}(\pi):=q(\pi)\prod_{j=1}^d (1+|s_\pi(j)|)
\end{equation}
We say that $L(s,\pi)$ belongs the the \emph{extended Selberg class} if it satisfies the following additional properties
\begin{enumerate}
   \setcounter{enumi}{3} \item There exists a positive $\delta<1/2$ such that $\abs{\alpha_j(p)} \leq p^\delta$ for all primes $p$ and $\abs{\Re s_\pi(j)} \leq \delta$.
\end{enumerate}
It is believed that any $L$-function as defined above satisfies the \emph{Generalized Riemann Hypothesis} \cite[Section 5.7]{IwKo}:
\begin{conj}[GRH]
Let $L(s,\pi)$ be an $L$-function. Then all zeroes in the critical strip $0<\Re s<1$ lies on the critical line $\Re s=1/2$.    
\end{conj}
This is known to imply the \emph{Generalized Lindel\"{o}f Hypothesis}.
\begin{conj}[GLH]
Let $L(s,\pi)$ be an $L$-function. Then for any $\varepsilon>0$, it holds that $L(1/2,\pi)\ll_\varepsilon \mathbf{c}(\pi)^\varepsilon$ where the implied constant is allowed to depend on the degree $d$.    
\end{conj}
\subsubsection{Automorphic facts}\label{sec:automorphicstuff}
The principal source of $L$-functions are automorphic forms. If $\pi$ is an automorphic representation of $\GL(d)/\Q$, then by the work of Langlands and Gelbart--Jacquet we have an $L$-function satisfying conditions (1)-(3) and by the work of Luo--Rudnick--Sarnak \todo{Should we add ref's here?}they also satisfy condition (4). In this paper, all $L$-functions are built through instances of Langlands functoriality from degree $2$ $L$-functions. In the sequel, we work with automorphic forms $\phi$ on $B^\times$ or $\GL(2)/\Q$ and, as is common practice, we denote by $L(s, \phi)$ the $L$-function $L(s, \pi)$, where $\pi$ is the $\GL(2)$ automorphic representation generated by $\phi$ or attached to the latter through the Jacquet-Langlands correspondence.
In other $L$-functions, one should similarly take $\phi$ to mean $\pi$. 

We recall the basics of the Jacquet--Langlands correspondence (see \cite[Chap.~10]{gelbart}, for example). Let $B$ be an indefinite quaternion algebra defined over $\Q$ of discriminant $D=\disc(B)$. Let $\pi^B$ be an automorphic representation  of $B^\times$ of level $q$ prime to $D$ with trivial central character such that the local constituent $\pi^B_\infty$ is a principal series representation. Then there exists an associated automorphic representation $\pi$ of $\PGL(2)/\Q$ of conductor $Dq$ such that the local constituent $\pi_{v}$ at a place $v$ is isomorphic to $\pi_{v}^B$ for all $v\nmid D$. We refer to $\pi$ as the \emph{Jacquet--Langlands transfer} of $\pi^B$. This induces a one-to-one correspondence between the non-constant Hecke--Maa{\ss} forms on the compact surface $Y_q$ and the cuspidal Hecke--Maa{\ss} forms on $Y_0(Dq)$ which are new at $D$ (see \cite{Strombergson01} for this classical interpretation).  

Recall that if $\pi$ is an automorphic representation of $\PGL(2)/\Q$, then the Dirichlet series coefficients $\lambda_\pi(n)$ are given by the eigenvalues of the Hecke operators defined above, properly normalized. If $\pi$ has trivial central character, then the Hecke eigenvalues are all real numbers \cite[eq.\ (6.3)]{DFI02}: 
\begin{equation}\label{eq:reallambda}
    \lambda_\phi(n)\in \R,\quad n\geq1,
\end{equation}
by self-adjointness of the Hecke operators. We have the following key bound due to Kim--Sarnak \cite[Appendix 2, Proposition 1]{Kim03} for $\pi$ unramified at~$p$:
\begin{equation}
    \label{eq:KimSarnak} |\alpha_\pi(p,j)|\leq p^{7/64}, 
\end{equation}
which implies that $|\lambda_\pi(p)|\leq 2 p^{7/64}$.
The same is true when $\pi$ is ramified at $p$ (see the classification and definition of local $L$-factors after Theorem 4.7.4 in \cite{bump}, keeping in mind that we only consider $\pi$ with trivial central character).
Similarly, if $\pi$ is unramified at $\infty$, then it holds that
\begin{equation}
    \label{eq:KimSarnak-infty} |\Re s_\pi(j)|\leq 7/64. 
\end{equation}

Recall that by \cite{GelbartJacquet78} and \cite{Kim03}, if $\pi$ is a cuspidal automorphic representation for $\GL(2)/\Q$, then there exists an automorphic representation $\sym^2 \pi$ for $\GL(3)/\Q$ and $\sym^4 \pi$ for $\GL(5)/\Q$, satisfying at unramified primes that
 \begin{equation}\label{eq:sym2sym4}
     \lambda_{\sym^2 \pi}(p)=\lambda_\pi(p^2),\quad\lambda_{\sym^4 \pi}(p)=\lambda_\pi(p^4). 
 \end{equation}
We also note that, if $\pi$ has trivial central character, then the adjoint representation $\ad \pi$ is isomorphic to $\sym^2 \pi$.\footnote{We will shortly state the Watson-Ichino formula. One of the $L$-functions commonly appearing in the main sources is $L(s, \ad \pi)$, for which we substitute $L(s, \sym^2 \pi)$ thanks to our observation above, as is done in other related works such as \cite{HumKhan20} or \cite{bingrong}.}
Finally, we recall the existence of the Rankin-Selberg $L$-function for two representations $\pi, \pi'$ of $\GL(n), \GL(m)$, where $\lambda_{\pi \otimes \pi'}(p) = \lambda_\pi(p) \cdot \lambda_{\pi'}(p)$.
The Satake parameters of a Rankin-Selberg convolution $\pi \otimes \pi'$ are given by the multiset of all the products between a Satake parameter of $\pi$ and a Satake parameter of $\pi'$. 
The Satake parameters of $\sym^k \pi$ are a subset of those of the convolution of $\pi$ with itself $k$ times.
 
We will need some analytic information about various $L$-functions.
 \begin{lemma}\label{lem:sym2sym4}
 Let $\pi$ be an automorphic representation for $\PGL(2)/\Q$ of square-free conductor. Then the $L$-functions $L(s,\sym^2 \pi)$ and $L(s,\sym^4 \pi)$ are holomorphic and belong to the extended Selberg class.
 \end{lemma}
 \begin{proof}
   From \cite{GelbartJacquet78} we know that $\sym^2 \pi$ is cuspidal exactly if $\pi$ is not \emph{monomial} (meaning that $\pi\cong \pi\otimes \eta$ for some Dirichlet character $\eta$). Since $\pi$ has square-free conductor and trivial central character, it is a standard fact that $\pi$ is not monomial, see e.g.\ the last sentence in \cite[Cor.\ 4.1.3]{Ramakrishnan2000}. For completeness we sketch an argument: by \cite{LabesseLanglands79} being monomial means exactly that $\pi$ is the automorphic induction $\pi(\chi)$ of a Hecke character $\chi$ of a quadratic field $K/\Q$.   But if $\chi$ is ramified at a finite prime of $K$ then the conductor of $\pi(\chi)$ at the corresponding rational prime is square-full and if  $\chi$ has conductor one then $\pi(\chi)$ has central character equal to the quadratic character associated to $K$ via class field theory. 
   Thus we conclude that $\sym^2\pi$ is indeed cuspidal.
   
   Furthermore, since $\sym^2\pi$ is self-dual, it follows by Rankin--Selberg theory (see \cite[Thm.~11.7.1]{getz-hahn}) that $L(s,\sym^2 \pi \otimes \sym^2 \pi)$ has a unique simple pole at $s=1$. 
   Using the Hecke relations (see Section \ref{sec:hecke-rel} below), it follows that at a prime $p$ the Dirichlet coefficient of the Rankin--Selberg $L$-function equals 
   \begin{align*}
       \lambda_{\sym^2 \pi}(p)^2&=\lambda_\pi(p^2)^2=\lambda_\pi(p^4)+\lambda_\pi(p^2)+1\\
       &=\lambda_{\sym^4 \pi}(p)+\lambda_{\sym^2 \pi}(p)+1. 
   \end{align*}
From strong multiplicity one and the above we conclude the factorization 
\begin{equation}
    L(s,\sym^2 \pi \otimes \sym^2 \pi)= L(s,\sym^4 \pi)L(s,\sym^2 \pi)\zeta(s),
\end{equation}
see e.g.\ \cite[third eq.,
p.\ 215]{RamaWang03}. Since $\zeta(s)$ has a simple pole at $s=1$ we conclude that indeed $L(s,\sym^4 \pi)$ is holomorphic. That the $L$-functions belong to the extended Selberg class follows from Rankin--Selberg theory \cite[Thm.~11.7.1]{getz-hahn} and the bounds towards Ramanujan (\ref{eq:KimSarnak}) and (\ref{eq:KimSarnak-infty}). 
 \end{proof}
\begin{rem}
    In fact, a more careful look at the local Langlands correspondence shows that, if $\pi$ has trivial central character and square-free conductor, then $\sym^2 \pi$ and $\sym^4 \pi$ are cuspidal.
    This follows from the argument that $\pi$ is not monomial (i.e. not of CM-type) and the observation in the paragraph before Corollary 2.9 in \cite{blomer-et-al}.
\end{rem}
 
\begin{lemma}\label{lem:automorphyRankinSelberg}
 Let $\pi_1,\pi_2,\pi_3$ be pairwise distinct automorphic representations for $\PGL(2)/\Q$ with trivial central characters and of square-free conductors. Then $L(s,\pi_1\otimes \pi_2)$,  $L(s,\sym^2 \pi_1\otimes \pi_2\otimes \pi_3)$, $L(s,\sym^2 \pi_1\otimes \sym^2 \pi_2)$ are holomorphic and belong to the extended Selberg class. 
\end{lemma}
\begin{proof}
It follows from \cite[Theorem M]{Ramakrishnan2000} that there exists a \emph{cuspidal} automorphic representation $\pi_1\otimes \pi_2$ of $\GL(4)/\Q$ such that $\lambda_{\pi_1\otimes \pi_2}(p)=\lambda_{\pi_1}(p)\lambda_{\pi_2}(p)$ for primes $p$ not dividing the conductors of $\pi_1,\pi_2$.
Here we are using that $\pi_1$ is not a twist of $\pi_2$, since the conductors are assumed square-free (see the cuspidality criterion in \emph{loc.\ cit.}). This yields the claim for the first $L$-function by general facts about automorphic $L$-functions. 

Next, our assumptions imply that $\sym^2 \pi_1 \not\cong  \sym^2 \pi_2$ (see \cite[Cor.~4.1.3]{Ramakrishnan2000}).
The claim for the last two $L$-function now follows from Rankin--Selberg theory (see \cite[Thm.~11.7.1]{getz-hahn}), as well as the Kim--Sarnak bound (\ref{eq:KimSarnak}) to obtain the bound for the Satake parameters $|\alpha_{\pi}(p,j)|\leq (p^{7/64})^4=p^{7/16}$ for $\pi\in \{\pi_1\otimes \pi_2, \sym^2 \pi_1\otimes \pi_2\otimes \pi_3, \sym^2 \pi_1\otimes\sym^2 \pi_2\}$ and similarly at the archimedean place.  
\end{proof}
Finally, we record the that if $\pi$ is cuspidal then
\begin{equation}\label{eq:positivesym2}
 L(1,\sym^2 \pi)>0,  
\end{equation} 
using the Rankin--Selberg method \cite[(5.101)]{IwKo}.

\subsubsection{Hecke relations} \label{sec:hecke-rel}
Let $\lambda(p^k)$ be the eigenvalue of an automorphic form on $Y_q$ with respect to the unramified Hecke operators $T_{p^k}$.
It is well-known (see \cite[(8.39)]{iwaniec-spectral}) that these satisfy the relations
\begin{equation} \label{eq:Heckerelation}
    \lambda(p^n) \lambda(p^m) = \sum_{j = 0}^{\min(m,n)} \lambda(p^{m+n-2j}).
\end{equation}
For instance, we often use that $\lambda(p)^2 = \lambda(p^2) + 1$ or that
\begin{equation} \label{eq:hecke-rel-third-pow}
    \lambda(p)^3 = \lambda(p^3) + 2 \lambda(p).
\end{equation}
Applying the formula iteratively, for every $d\geq 1$, we may find coefficients $\alpha^{(d)}_{j,n}\in \Z$ that satisfy
\begin{equation} \label{eq:powers-of-lambda}
    \lambda(p^d)^n=\sum_{j=0}^{dn}\alpha^{(d)}_{j,n}\lambda(p^j).
\end{equation}

Recall also that the Hecke eigenvalue is given in terms of the two Satake parameters $\alpha(p,1), \alpha(p,2)$ through
\begin{displaymath}
    \lambda(p) = \alpha(p,1) + \alpha(p,2).
\end{displaymath}
Thus, the $L$-factor at $p$ is
\begin{equation} \label{eq:l-factor-at-p}
    (1 - \alpha(p,1) p^{-s})^{-1} \cdot (1 - \alpha(p,2) p^{-s})^{-1} = (1 - \lambda(p) p^{-s} + p^{-2s})^{-1},
\end{equation}
since $\alpha(p,1) \alpha(p,2) = 1$.

Another standard use of the Hecke relations that can be found in our work is to compute power sums like $\alpha(p,1)^k + \alpha(p,2)^k$.
For this, we may use Newton's identities and the facts above to express the latter as a polynomial in $\lambda(p)$. 
If necessary, this polynomial can be expressed as a linear combination of $(\lambda(p^j))_j$ by using \eqref{eq:Heckerelation} iteratively to compute each power of $\lambda(p)$.
For instance,
\begin{displaymath}
    \alpha(p,1)^2 + \alpha(p,2)^2 = \lambda(p)^2 - 2 = \lambda(p^2) - 1.
\end{displaymath}

To study the combinatorics of the Hecke relations further, it is useful to consider the following generating functions.
The standard Chebyshev polynomials of the second kind are defined by
\begin{displaymath}
    (1 - 2tX + X^2)^{-1} = \sum_{n \geq 0} U_n(t) X^n.
\end{displaymath}
It is well-know and easy to see that $U_0(t) = 1$ and $U_1(t) = 2t$.
Due to the relations \eqref{eq:Heckerelation} and \eqref{eq:l-factor-at-p}, we may substitute $t \mapsto \lambda(p)/2$ and this gives $\lambda(p^n) = U_n(\lambda(p)/2)$.
The $U_n$ thus also satisfy the analogue of~\eqref{eq:powers-of-lambda}, meaning that
\begin{displaymath}
    U_d(t)^n=\sum_{j=0}^{dn}\alpha^{(d)}_{j,n}U_j(t).
\end{displaymath}

\begin{lemma}\label{lem:Heckerelations}
For the coefficients given in \eqref{eq:powers-of-lambda}, it holds that
\begin{enumerate}
\item  \label{item:hecke2}$\alpha^{(d)}_{j,n}\geq 0$,
\item \label{item:hecke4} $\sum_{j=0}^{dn} \alpha^{(d)}_{j,n}\leq (d+1)^{n-1}$,
\item \label{item:hecke1} $\alpha^{(d)}_{0,1}=0$,
\item \label{item:hecke3} $\alpha^{(1)}_{0,n}=\delta_{2|n}\frac{n!}{(\frac{n}{2})!(\frac{n}{2}+1)!}$.
\end{enumerate}
\end{lemma}
\begin{proof}
    The Hecke relations (\ref{eq:Heckerelation}) imply that 
    $$\alpha^{(d)}_{j,n+1}=\sum_{i=\max(0,\lceil \frac{d-j}{2} \rceil)}^d \alpha^{(d)}_{j-d+2i,n}.$$
    In particular, we conclude that $\alpha^{(d)}_{j,n}\geq 0$ and
    $$ \sum_{j=0}^{d(n+1)} \alpha^{(d)}_{j,n+1}\leq (d+1)  \sum_{j=0}^{dn} \alpha^{(d)}_{j,n}\leq \ldots \leq (d+1)^n \alpha^{(d)}_{d,1} = (d+1)^n,$$
    yielding the second property.
    
    The third property is obvious. 
    The last one can be deduced, for example, using the orthogonality of Chebyshev polynomials.
    More precisely, recall that $(U_n(t))_n$ is an orthonormal system in $L^2([-1,1])$ with respect to the measure given by $\frac{2}{\pi} \sqrt{1 - x^2} \, dx$. 
    Therefore,
    \begin{displaymath}
        \alpha_{0,n}^{(1)} = \frac{2}{\pi} \int_{-1}^1 U_1(x)^n \sqrt{1 - x^2} \, dx.
    \end{displaymath}
    Recalling that $U_1(x) = 2x$ and substituting $x \mapsto \sin(x)$, we obtain
    \begin{displaymath}
        \alpha_{0,n}^{(1)} = 2^n \cdot \frac{2}{\pi} \int_{-\pi/2}^{\pi/2} \sin(x)^n \cos(x)^2 \, dx.
    \end{displaymath}
    Since $\sin(x)$ is odd, the integral vanishes if $2 \nmid n$.
    Next, we recall the well-known fact (the following is one of Wallis' integrals) that
    \begin{displaymath}
        \int_{-\pi/2}^{\pi/2} \sin(x)^{2k} \, dx = 2 \cdot \frac{\pi}{2} \cdot \frac{1}{2^{2k}} \cdot \frac{(2k)!}{(k!)^2}.
    \end{displaymath}
    Writing $\cos(x)^2 = 1 - \sin(x)^2$ and applying these formulae, we obtain that
    \begin{displaymath}
        \alpha_{0,2k}^{(1)} = 2 \left( \frac{(2k)!}{(k!)^2} - \frac14 \cdot \frac{(2k+2)!}{(k+1)!^2} \right).
    \end{displaymath}
    The claim follows by a simple calculation.
\end{proof}

We record here a consequence of the Hecke relations that plays an important technical role in this paper.
\begin{lemma}\label{lem:variance-coeffs}
    For $\pi, \sigma_1, \sigma_2$ automorphic representations of $\PGL(2)/\Q$, unramified at a prime $p$, we have
    \begin{align*}
        (\lambda_\pi(p)^2 + \lambda_{\sigma_1}(p) \lambda_{\sigma_2}(p))^2 &= 3 + \lambda_{\sym^4 \pi}(p) + 3 \lambda_{\sym^2 \pi}(p) +  \lambda_{\sym^2 \sigma_1 \otimes \sym^2 \sigma_2}(p) \\ 
        &+ \lambda_{\sym^2 \sigma_1}(p) + \lambda_{\sym^2 \sigma_2}(p) + 2 \lambda_{\sym^2 \pi \otimes \sigma_1 \otimes \sigma_2}(p) + 2 \lambda_{\sigma_1 \otimes \sigma_2}(p),
    \end{align*}
    if $\sigma_1 \neq \sigma_2$ and
    \begin{align*}
        (\lambda_\pi(p)^2 + \lambda_{\sigma}(p)^2)^2 &= 6 + \lambda_{\sym^4 \pi}(p) + 5 \lambda_{\sym^2 \pi}(p) + 2 \lambda_{\sym^2 \pi \otimes \sym^2 \sigma}(p) \\ 
        &+ \lambda_{\sym^4 \sigma}(p) + 5 \lambda_{\sym^2 \sigma}(p),
    \end{align*}
    if $\sigma_1 = \sigma_2 = \sigma$.
\end{lemma}
\begin{proof}
    Note that $\lambda_\pi(p)^2 = \lambda_\pi(p^2) + 1$ and insert this into the left-hand side.
    Squaring a trinomial, we are done by applying the relation $\lambda_\pi(p^2)^2 = \lambda_\pi(p^4) + \lambda_\pi(p^2) + 1$.
\end{proof}

\subsubsection{Local representations and triple product $L$-functions}Let $\pi_1^B, \pi_2^B, \pi_3^B$ be three cuspidal automorphic representations of $B^\times$ with trivial central character and Jacquet--Langlands transfers $\pi_1,\pi_2,\pi_3$. The $L$-functions most central to our equidistribution problem are the triple product $L$-function $L(s, \pi_1 \otimes \pi_2 \otimes \pi_3)$.
For our purposes we will need  more precise information about the Satake parameters, which we will now recall. 
The triple product $L$-function can be defined in a more abstract way using the local Langlands correspondence (see Chapter 3 of Watson's thesis \cite{watson}, for example), but we write them down explicitly in the cases relevant to us.

We assume from now on that $\pi_{i, \infty}$ is a principal series representation $\mathcal{P}(s_i, \delta_i)$ of parameter $s_i \in \C$ and sign given by $\delta_i \in \{0,1\}$, that is, the representation $\sgn^{\delta_i} \abs{\cdot}^{s_i} \boxtimes \sgn^{\delta_i} \abs{\cdot}^{-s_i}$.
It contains a spherical vector on which the Casimir operator acts with eigenvalue $1/4 - s_i^2$.
Tempered representations have $s_i = it_i$ with $t_i \in \R$.
More generally, we have $s_i \in i\R \cup [-7/64, 7/64]$, by the Kim--Sarnak bound \eqref{eq:KimSarnak-infty}.
If $\pi$ is attached to an automorphic form $\phi$, then we denote the parameter of $\pi_\infty$ by $s_\phi$.

Furthermore, assume that $\pi_{i,p}$ is unramified for all primes $p$ except at $p \mid \disc(B)$ and, potentially, at $p = q$.
Thus, at unramified primes $p$, there is an unramified character $\chi_i:\Q_p \to \C$ such that $\pi_{i,p}$ is isomorphic to $\chi_i \boxtimes \chi_i^{-1}$.
The Hecke operator $T_p$ acts on the spherical vector with eigenvalue $\chi_i(p) + \chi^{-1}_i(p)$.
We denote $\alpha_{i,p} = \chi_i(p)$.

At the prime $q$, either $\pi_{i,q}$ is unramified as above and we keep the notation, or we assume that $\pi_{i,q}$ has level $q$.
Thus, in the latter case, it is a special representation $\sigma_{\chi_i} = \chi_i \cdot \Stein$, a quadratic twist of the Steinberg representation.
Here, $\chi_i$ is an unramified quadratic character determined by the value $\chi_i(q) \in \{ \pm 1\}$.
For concreteness, we note that the Fricke involution $\pi_{i,q} \begin{psmallmatrix}
    & -1 \\ q &
\end{psmallmatrix}$ acts on the standard newvector in this representation with eigenvalue $-\chi_i(q)$.

At a prime $p$ dividing $\disc(B)$, the local representation $\pi_{i,p}^B$ is one-dimensional.
The Jacquet-Langlands correspondence determines the representation $\pi_{i,p}$ to be a special representation as above (see \cite[Sec. 1.4]{watson} for more details).

The local $L$-factors are defined as follows (see \cite[Sec. 3.1]{watson} and \cite[Sec. 4.1]{woodbury}).
For a prime $p$, if $\pi_{i,p}$ is unramified for all $i$, then
\begin{equation}\label{eq:tripleunram}
    L_p(s, \otimes_i \pi_{i,p}) = \prod_{ \{\pm\}^3 } (1-\alpha_{1,p}^\pm \alpha_{2,p}^\pm \alpha_{3,p}^\pm p^{-s})^{-1}.
\end{equation}
If $\pi_{1,p}$ and $\pi_{2,p}$ are unramified and $\pi_{3,p} = \sigma_{\chi_3}$, then
\begin{equation}\label{eq:tripleram1}
    L_p(s, \otimes_i \pi_{i,p}) = \prod_{ \{\pm\}^2 } (1-\alpha_{1,p}^\pm \alpha_{2,p}^\pm \chi_3(p) p^{-s-1/2})^{-1}.
\end{equation}
If $\pi_{1,p}$ is unramified, $\pi_{3,p} = \sigma_{\chi_2}$ and $\pi_{3,p} = \sigma_{\chi_3}$, then
\begin{equation}\label{eq:tripleram2}
    L_p(s, \otimes_i \pi_{i,p}) = \prod_{ \{\pm\} } (1-\alpha_{1,p}^\pm \chi_2(p) \chi_3(p) p^{-s})^{-1}(1-\alpha_{1,p}^\pm \chi_2(p) \chi_3(p) p^{-s-1})^{-1}.
\end{equation}
If all representations are special, $\pi_{i,p} = \sigma_{\chi_i}$, then
\begin{equation}\label{eq:tripleram3}
    L_p(s, \otimes_i \pi_{i,p}) = (1- \chi_1(p) \chi_2(p) \chi_3(p) p^{-s-1/2})^{-2}(1- \chi_1(p) \chi_2(p) \chi_3(p) p^{-s-3/2})^{-1}.
\end{equation}

At the archimedean place, let
\begin{displaymath}
    \Gamma_{\R}(s) = \pi^{-s/2} \Gamma(s/2).
\end{displaymath}
Since we are considering only the principal series $\mathcal{P}(s_i, \delta_i)$, we put $\delta \equiv \sum_i \delta_i \in \{0,1\}$ and record that
\begin{displaymath}
    L_\infty(s, \otimes_i \pi_{i,\infty}) = \prod_{\{\pm\}^3} \Gamma_{\R}(s \pm s_1 \pm s_2 \pm s_3 + \delta).
\end{displaymath}

We also need the adjoint or symmetric square $L$-function.
At a prime $p$, if $\pi_{i,p}$ is unramified,
\begin{equation}\label{eq:localsym2}
    L_p(s, \sym^2 \pi_{i,p}) = \left[ (1 - \alpha_{i,p}^2 p^{-s}) (1- p^{-s}) (1 - \alpha_{i,p}^{-2} p^{-s}) \right]^{-1},
\end{equation}
and, in the ramified case,
\begin{displaymath}
    L_p(s, \sym^2 \sigma_{\chi_i}) = (1 - p^{-s-1})^{-1}.
\end{displaymath}
At the archimedean place,
\begin{displaymath}
    L_\infty(s, \sym^2 \pi_{i, \infty}) = \Gamma_\R (s + 2s_i) \Gamma_\R (s) \Gamma_\R (s - 2s_i).
\end{displaymath}

All $L$-functions denoted by $L(s, \ldots)$ are now defined as the product of the local $L$-functions over the finite places, that is,
\begin{displaymath}
    L(s, \ldots) = \prod_{v < \infty} L_v(s, \ldots),
\end{displaymath}
and we denote the completed $L$-functions\footnote{The literature on Ichino's formula and its applications has inconsistent notation. We note that \cite{ichino} and \cite{woodbury} use $L(s, \ldots)$ to denote the completed $L$-function, while \cite{nelson} uses our notation. Watson's notation \cite{watson} is also slightly different.} by
\begin{displaymath}
    \Lambda(s, \ldots) = L_\infty(s, \ldots) \cdot L(s, \ldots).
\end{displaymath}

\subsection{Watson-Ichino formula}

We now translate the inner products involving three forms appearing in \eqref{eq:spectral-expansion} into $L$-functions using the Watson-Ichino formula, making the relevant local constants explicit. 

\begin{prop} \label{prop:watson-ichino}
    Let $\phi_1, \phi_2, \phi_3$ be cuspidal Hecke--Maaß forms on $Y_1$ or $Y_q$.
    Then 
    \begin{displaymath}
        \frac{\left| \int_{Y_q} \phi_1 \phi_2 \bar \phi_3 \, d\mu_q \right|^2}{\prod_i \norm{\phi_i}^2} 
        = \frac{\zeta(2) C_\infty C_q C_B}{2} \cdot \frac{\Lambda(1/2, \phi_1 \times \phi_2 \times \phi_3)}{\prod_i \Lambda(1, \sym^2 \phi_i)},
    \end{displaymath}
    where
    \begin{displaymath}
        C_q = \begin{cases}
            1 & \text{for $\phi_1, \phi_2, \phi_3$ on $Y_1$}; \\
            \frac1q & \text{for $\phi_1$ on $Y_1$ and $\phi_2, \phi_3$ newforms on $Y_q$};\\
            \frac{(1+q^{-1})(1+\eta)}{q} & \text{for $\phi_1, \phi_2, \phi_3$ newforms with respective} \\
            & \text{Fricke eigenvalue $\eta_i$ and $\eta = \eta_1 \eta_2 \eta_3$};
        \end{cases}
    \end{displaymath}
    and
    \begin{displaymath}
        C_\infty = \frac{1 - \delta}{2}
    \end{displaymath}
    for $\delta \equiv \sum_i \delta_i$, $\delta \in \{0,1\}$, where $(-1)^{\delta_i}$ is the eigenvalue of $z \mapsto -\bar z$ acting on $\phi_i$.
    The constant $C_B$ depending on the quaternion algebra $B$ and the automorphic forms above, satisfying $C_B \ll \disc(B)^{-1}$ with absolute implied constant.

    If $\phi_1, \phi_2, \phi_3$ are cuspidal Hecke--Maaß forms, $\phi_1, \phi_2$ on $Y_1$ and $\phi_3$ a newform on $Y_q$, then
    \begin{displaymath}
        \frac{\left| \int_{Y_q} \phi_1 l_q(\phi_2) \bar \phi_3 \, d\mu_q \right|^2}{\prod_i \norm{\phi_i}^2} 
        = \frac{\zeta(2) C_\infty C_B}{2(q+1)} \cdot \frac{\Lambda(1/2, \phi_1 \times \phi_2 \times \phi_3)}{\prod_i \Lambda(1, \sym^2 \phi_i)},
    \end{displaymath}
    for some constant $C_B$ with the same properties as above.
\end{prop}
\begin{proof}
    We adelize the newforms $\phi_1, \phi_2, \phi_3$ (see, for instance, \cite[Chap. 5]{gelbart}) and generate representations $\pi_1^B, \pi_2^B, \pi_3^B$, respectively.
    Denote by $\phi'_2$ the shifted vector corresponding to $l_q \phi_2$ inside $\pi_2^B$. 
    If we let $[B^\times]$ denote the automorphic quotient $B^\times(\Q) \cdot Z(\A) \backslash B^\times(\A)$, where $Z$ is the center of $B^\times$, the Watson-Ichino formula, \cite[Thm. 1.1]{ichino} or \cite[Thm. 2.1]{woodbury}, implies that
    \begin{displaymath}
        \frac{\abs{\int_{[B^\times]} \phi_1 \phi_2 \bar \phi_3 \, db}^2}{\prod_i \int_{[B^\times]} \abs{\phi_i}^2 \, db} = \frac{\zeta(2)}{8} \cdot \frac{\Lambda(1/2, \phi_1 \times \phi_2 \times \phi_3)}{\prod_i \Lambda(1, \sym^2 \phi_i)} \cdot \prod_v C_v,
    \end{displaymath}
    where $db$ is the Tamagawa measure, as defined in \cite[Sec.~2.2]{woodbury}, and $C_v$ are local constants.
    The analogous formula holds when replacing $\phi_2$ by $\phi_2'$ and we note that the representations these two vectors generate are the same.

    Before computing the local constants, we translate to the classical integrals.
    Let $\Gamma$ be either $\Gamma_1$ or $\Gamma_q$ and $K \subset B^\times(\A)$ be such that
    \begin{displaymath}
        \Gamma \backslash \mathbb{H} \cong [B^\times]/K.
    \end{displaymath}
    The invariant probability measure $\mu_{\text{prob}}$ on $\Gamma \backslash \mathbb{H}$ is given through this isomorphism by the Tamagawa measure normalized by the volume of the space.
    We have $\vol([B^\times]/K) = \vol([B^\times])/\vol(K)$ and, as in \cite[Sec.~2.2]{woodbury}, the Tamagawa number of $PB^\times$ is $\vol([B^\times]) = 2$.
    We have the formal computation
    \begin{displaymath}
        \int_{\Gamma \backslash \mathbb{H}} \mu_{\text{prob}} = \vol([B^\times]/K)^{-1} \int_{[B^\times]/K} db = \frac{\vol(K)}{2} \int_{[B^\times]/K} db = \frac12 \int_{[B^\times]} db.
    \end{displaymath}
    
    The local constant at the prime $q$ was computed by several authors.
    For the first assertion in the proposition's statement, the computations can be found in the work of Watson \cite{watson} (see \cite[Prop.~4.4]{woodbury} for ease of comparison) and Nelson \cite{nelson} (especially for $\phi_1$ of level $1$ and $\phi_2, \phi_3$ of level $q$; see also \cite[Rem.~4.2]{nelson} and \cite[Prop.~4.3]{woodbury}).
    The constants at primes dividing $\disc(B)$ are also computed in the sources cited, e.g. \cite[Prop.~4.5]{woodbury}, and they are of the form $\delta \cdot p^{-1}(1-p^{-1})$ for $p \mid \disc(B)$ and $\delta \in \{0,1\}$.
    
    For the second assertion, we use the computation in \cite[Cor.~4.2]{woodbury}, which is valid in our case since $B$ is unramified at $q$.
    The only difference is the potential conjugation of the standard Eichler order.
    More precisely, we take the isomorphism $\Oc \cong \M_2(\Z_p)$ discussed in Section \ref{sec:hecke} and note that $a_q = \xi_1 \diag(q, 1) \xi_2$ for $\xi_1, \xi_2 \in \SL_2(\Z_p)$ and $\Oc_0(q)^\times_q = \xi_2^{-1} K_0 \xi_2$, where $K_0$ are the units of the standard Eichler order (we use the same notation as in \cite[Sec.~3]{woodbury}).
    Now $\phi_{3,q}$ is the newvector invariant under $\Oc_0(q)^\times_q$ and, thus, $\pi_{3,q}(\xi_2)\phi_{3,q}$ is the standard newvector invariant under $K_0$.
    When computing the matrix coefficient integral that gives the local constant, we make a change of variables $g \xi_2^{-1} \mapsto h$ and bring it back to the setting of \cite[Cor.~4.2]{woodbury}, using that $\phi_1$ is invariant under $\xi_2$ and that $\phi_2$ is invariant under $\xi_1$.    
\end{proof}

\begin{rem} \label{rem:trivial-vanishing}
    It is clear that $\int_{Y_q} \phi_1 \phi_2 \bar \phi_3$ vanishes if the product of the Fricke eigenvalues of $\phi_1$, $\phi_2$, $\phi_3$ is $-1$, as the formula in the previous proposition also implies.
    Similarly, the automorphism of $Y_q$ given by $z \mapsto -\bar z$ acts on $\phi_i$ by $(-1)^{\delta_i}$, where $\delta_i$ is the sign of the archimedean principal series representation in the decomposition of the associated automorphic representation of $\phi_i$.
    Consequently, if $\sum \delta_i$ is odd, then the triple product integral also vanishes, fact which is recorded in the constant $C_\infty$.
    For the rest of this paper, we assume that $\sum \delta_i $ is even.
\end{rem}

Preparing to insert the Watson-Ichino formula into the spectral expansion, we gather here bounds for the archimedean part of the $L$-function expression above.

\begin{lemma}\label{lem:stirling}
    Let $\phi_1,\phi_2,\phi_3$ be non-constant Hecke--Maaß forms on $Y_1$ or $Y_q$, as in any of the cases of Proposition~\ref{prop:watson-ichino}, with parameters $s_j = \sigma_j + it_j$, where we may assume that $t_j, \sigma_j \geq  0$, and signs $\delta_j \in \{0, 1\}$.
    Assume that $\sum_j \delta_j$ is even.
    Then it holds that
    \begin{align*}
        &\frac{L_\infty(1/2,\phi_1\times \phi_2\times \phi_3)}{\prod_iL_\infty(1,\sym^2 \phi_i)} \\
        &\ll \exp\left(-\frac{\pi}{2}\left( \sum_{\epsilon_1,\epsilon_2\in\{\pm 1\}} |t_{1}+\epsilon_1 t_{2}+\epsilon_2 t_{3}| -2\sum_it_{i} \right)\right) \\
        &\cdot \prod_{\epsilon_1,\epsilon_2\in\{\pm 1\}} (|t_{1}+\epsilon_1 t_{2}+\epsilon_2 t_{3}|+1)^{-1/2}.
    \end{align*}
\end{lemma}
\begin{proof}
    Recall, by the classification of admissible principal series, that $s_j \in i\R \cup [-7/64, 7/64]$.
    By duality, we may also swap $s_j$ for $-s_j$ and we can thus assume that $\sigma_j, t_j \geq 0$.
    
    By observing that $\Gamma_{\R}(s) \asymp \Gamma(s/2)$ for $\Re(s) \ll 1$, we recall from Section \ref{sec:l-functions} (see also Remark \ref{rem:trivial-vanishing}) that
    $$L_\infty(s,\phi_1 \times \phi_2 \times \phi_3) \asymp \prod_{\{\pm 1\}^3} \Gamma\left(\frac{s+\sum_{j} \pm s_{j}}{2}\right),$$
    and 
    $$L_\infty(s,\sym^2 \phi) \asymp \Gamma\left(\frac{s+2s_{\phi}}{2}\right)\Gamma\left(\frac{s-2s_{\phi}}{2}\right)\Gamma\left(\frac{s}{2}\right).$$
    In the case that all representations are tempered, i.e. $\sigma_i = 0$ for all $i$, the result follows directly from Stirling's approximation \cite[eq.~(B.7)]{iwaniec-spectral}, namely
    $$\Gamma(\sigma+it) \asymp t^{\sigma-1/2} e^{-\pi t/2}$$
    for large $t>0$, which holds uniformly in vertical strips of bounded width.
    When $t$ is small, we replace $t^{\sigma-1/2}$ with $(1+t)^{\sigma - 1/2}$ and take care that $\sigma$ is not too close to a pole of the Gamma function.
    If, for instance, $\phi_1$ is tempered, but $\phi_2$ or $\phi_3$ is non-tempered, then we may still use Stirling's formula where adequate, and, otherwise, the bounds towards Selberg's conjecture ensure that we stay away from any poles of the Gamma function.
    In the triple product, the exponents in the factors $t^{\sigma - 1/2}$ sum up and, since the signs alternate, $\sigma_2$ or $\sigma_3$ do not contribute.
    If all $\phi_i$ are non-tempered, the bounds towards Selberg imply a uniform bound since $1/2 - \sum \sigma_i > 0.17$.
\end{proof}

\begin{rem}
    As observed in the proof of the preceding lemma, the spectral parameter $s_\phi = \sigma_\phi + i t_\phi$ of an Maaß form $\phi$ can be assumed to have $t_\phi \geq 0$.
    We assume this in the rest of the paper for all automorphic forms appearing.
\end{rem}

\section{The spectral expansion and preliminary reductions}
We now restrict to $B/\Q$ being an indefinite non-split quaternion algebra. 
This means that, for $q$ coprime to $\disc(B)$, the surface $Y_q$ defined in Section \ref{sec:hecke} is compact, and so the Laplacian on $Y_q$ has pure point spectrum. 
In the sequel, we will restrict to $q$ being prime.

Let $F$ be a Hecke--Maaß newform on $Y_q$ with spectral parameter $s_F = \sigma_F + it_F$ and denote by $\mu_F=|F|^2\mu_q$ the associated probability measure.
The goal is to prove effective equidistribution on $Y_1\times Y_1$ of the push-forward measure $(\iota_q)_\ast \mu_F$ as $q\rightarrow \infty$. 
Our approach is spectral and will rely on the following version of Weyl's criterion. Let $\Bc(Y_1)$ be a fixed orthonormal basis of Hecke--Maaß forms for $Y_1$. 
\begin{lemma}[Weyl's criterion]\label{lem:Weyl}
Let $\nu_n$ be a sequence of probability measures on $Y_1\times Y_1$. Let $E:\Z_{>0}\rightarrow \R_{\geq 0}$ be a function such that $E(n)\rightarrow 0$ as $n\rightarrow \infty$. Assume that there exists a constant $A>0$ such that for $f_1,f_2\in \Bc(Y_1)$ which are not both constant,
\begin{equation*}
    \nu_n(f_1\otimes f_2)\ll (t_{f_1}+1)^{A}(t_{f_2}+1)^{A}E(n). 
\end{equation*}
Then $\nu_n$ equidistributes on $Y_1\times Y_1$ as $n\rightarrow \infty$ with effective rate $E(n)$.
Explicitly, for any smooth function $\psi: Y_1\times Y_1\rightarrow \C$ it holds that
\begin{equation*}
    \nu_n(\psi)=(\mu_1\otimes \mu_1)(\psi)+O(\mathcal{S}_A(\psi)E(n)),
\end{equation*}
as $n\rightarrow \infty$, where $\mathcal{S}_A$ denotes an appropriate Sobolev norm depending on $A$. 
\end{lemma}
\begin{proof}
  The proof is quite standard and we will simply sketch the arguments. 
  We spectrally expand
  \begin{align*}
    \nu_n(\psi)&=\sum_{f_1,f_2\in \Bc(Y_1)} \langle \psi,f_1\otimes f_2\rangle_{\mu_1\otimes \mu_1} \nu_n(f_1\otimes f_2)\\
    &=(\mu_1\otimes \mu_1)(\psi)+O\left(E(n)\sum_{\substack{f_1,f_2\in \Bc(Y_1),\\ (f_1,f_2)\neq (1,1)}} \langle \psi,f_1\otimes f_2\rangle_{\mu_1\otimes \mu_1} (t_{f_1}+1)^A(t_{f_2}+1)^A  \right),  
  \end{align*}
  using that $\mu_n$ are probability measures to pick out the contribution from $f_1=f_2=1$. 
  Now, by self-adjointness of the Laplacian $\Delta$, we have for any $B>1$ that
  $$\langle \psi,f_1\otimes f_2\rangle_{\mu_1\otimes \mu_1} = \lambda_{f_1}^{-B}\lambda_{f_2}^{-B}\langle (\Delta\otimes \Delta)^B\psi,f_1\otimes f_2\rangle_{\mu_1\otimes \mu_1}.  $$
  Clearly $|\!|(\Delta\otimes \Delta)^B\psi|\!|_\infty\leq \mathcal{S}_B(\psi)$ where $\mathcal{S}_B$ denotes the $L^\infty$-Sobolev norm depending on the $2B$ first derivatives.
  Picking $B$ sufficiently large and using the Weyl law, we conclude.
  Note that we may approximate $\lambda_{f_1}$ by $1 + t_{f_1}^2 \asymp (1 + t_{f_1})^2$ since we have uniform bounds for non-tempered spectral parameters and a uniform spectral gap \eqref{eq:KimSarnak-infty}, for example.
\end{proof}
In the setting of the above lemma, we refer to $\nu_n(f_1\otimes f_2)$ as the \emph{Weyl sums} of the equidistribution problem. In our setting we will apply the Weyl criterion with error-term $(\log q)^{-1/4+\varepsilon}$ for any $\varepsilon$.

Note first of all that, if exactly one of  $f_1,f_2\in \Bc(Y_1)$ is constant, then we can easily bound the Weyl sums under the assumption of GLH. 
By Remark \ref{rem:fricke-w-invariance} regarding the Fricke involution, we may switch $f_1$ and $f_2$ and thus assume that $f_2$ is the constant function.
We are now in the case of the Kowalski--Michel--VanderKam conjecture.
\begin{prop} \label{prop:que-level-glh}
   Under GLH, for any $\varepsilon>0$, we have
   \begin{equation*}
       \langle |F|^2, f_1\rangle _q \ll_{\varepsilon, f_1} q^{-1/2+\varepsilon},
   \end{equation*}
   where the implied constant depends polynomially on the eigenvalue of $f_1$.
\end{prop}
\begin{proof}
    This follows directly from the Watson-Ichino formula, Proposition~\ref{prop:watson-ichino}, bounds for archimedean $L$-factors, Lemma~\ref{lem:stirling}, and a theorem of Hoffstein--Lockhart \cite{hoffstein-lockhart} that asserts $L(1, \sym^2 F)^{-1} \ll (\lambda_F q)^{\varepsilon}$ and $L(1, \sym^2 f_1) \ll \lambda_{f_1}^\varepsilon$.
    For the archimedean factors, we note here that 
    \begin{displaymath}
        \sum_{\{\pm\}^2}\abs{t_F \pm t_F \pm t_{f_1}} - 4t_F - 2t_{f_1} = 2\max( t_\psi-2t_F,0) \geq 0,
    \end{displaymath}
    and we observe that $L(1, \sym^2 F)^{-1} \cdot (\abs{2t_F + t_{f_1}})^{-1/2} \ll q^{\varepsilon}$.
    Thus, the dependence on $t_F$ is removed.
\end{proof}

Thus, for the rest of the paper we may assume that both $f_1,f_2\in \Bc(Y_1)$ are non-constant. The key idea is to consider the spectral expansion
\begin{equation} \label{eq:spectral-expansion}
    \inner{f_1 \cdot (l_q f_2), \abs{F}^2}_q = \sum_{\varphi \in \Bc(Y_q)} \inner{f_1 \cdot (l_q f_2), \varphi}_q \cdot \inner{\varphi, \abs{F}^2}_q,
\end{equation}
where the sum is over an orthonormal basis $\Bc(Y_q)$ of Hecke--Maaß forms for~$Y_q$. 
This allows us to relate all the non-constant terms in the sum above to $L$-functions by Proposition~\ref{prop:watson-ichino}. 
We will now proceed to give some preliminary bounds and reductions.

\subsection{The old spectrum}
We first sieve out the old spectrum in the expansion \eqref{eq:spectral-expansion} above.

\begin{lemma}
    An orthonormal basis for the space of old forms on $Y_q$, i.e. those orthogonal to the newforms, for $q$ prime is given by 
    \begin{displaymath}
        \Bc_{\mathrm{old}}(Y_q) = \{\psi_l : \psi \in \Bc(Y_1),\, l \mid q \},
    \end{displaymath}
    where, if $\psi$ is not the constant function, then
    \begin{displaymath}
        \psi_1 = \psi \quad \text{and} \quad \psi_q = \frac{1}{\sqrt{1 - \tilde{\lambda}_\psi(q)^2}}(l_q\psi - \tilde{\lambda}_\psi(q) \cdot \psi),
    \end{displaymath}
    where $\tilde{\lambda}_\psi(q)$ is the $\tilde{T}_q$ eigenvalue of $\psi$.
    If $\psi$ is the constant function, then we set $\psi_1 = \psi_q = \psi$.
\end{lemma}
\begin{proof}
    This follows from the fact that the set $\{\psi, l_q\psi \mid \psi \in \Bc(Y_1)\}$ spans the space of old forms (see \cite[Sec.~8.5]{iwaniec-spectral}, for instance -- the theory for congruence subgroups applies here since these are local properties) and the Gram-Schmidt process.
    Indeed, we start with $\psi_1$, which is a normalized vector by Remark \ref{rem:norms}.
    To add $l_q \psi$, we simply recall Lemma~\ref{lemma:hecke-corr-op} which computes that
    \begin{displaymath}
        \inner{\psi, l_q \psi}_q = \tilde{\lambda}_\psi(q).
    \end{displaymath}
    We finish by applying the same result to compute that
    \begin{align*}
        \norm{l_q\psi - \tilde{\lambda}_\psi(q) \cdot \psi}^2 &= \norm{l_q \psi} - \tilde{\lambda}_\psi(q)^2\\
        &= \inner{l_q(\psi \bar{\psi}), 1}_q - \tilde{\lambda}_\psi(q)^2 \\
        &= \inner{\psi \bar{\psi}, \tilde{T}_q 1}_q - \tilde{\lambda}_\psi(q)^2 \\
        &= 1 - \tilde{\lambda}_\psi(q)^2,
    \end{align*}
    using that the Hecke operator is self-adjoint and that it acts on constant functions with eigenvalue $1$.
\end{proof}

\begin{lemma} \label{lemma:old-spectrum-negligible}
    Under GLH, the old spectrum is negligible.
    More precisely,
    \begin{displaymath}
        \sum_{\varphi \in \Bc_{\mathrm{old}}(Y_q)} \inner{f_1 \cdot (l_q f_2), \varphi}_q \cdot \inner{\varphi, \abs{F}^2}_q \ll_{f_i}  q^{\theta + \varepsilon}
    \end{displaymath}
    where $\theta=\frac{7}{64}-\frac{1}{2}$ is the current best bound towards Ramanujan.
    Here the implied constant depends polynomially on the analytic conductor of $f_1$ and $f_2$.
\end{lemma}
\begin{proof}
    Consider the contribution of the basis elements $\psi_1 \in \Bc_{\mathrm{old}}(Y_q)$ in \eqref{eq:spectral-expansion} for $\psi \in \Bc(Y_1)$.
    By Lemma~\ref{lemma:hecke-corr-op}, we have
    \begin{displaymath}
        \inner{f_1 \cdot l_q f_2, \psi}_q = \inner{l_q f_2, \bar{f}_1 \psi}_1 =  \tilde{\lambda}_{f_2}(q) \inner{f_1 f_2, \psi}_1
    \end{displaymath}
    and therefore,
    \begin{align*}
        \sum_{\psi \in \Bc(Y_1)} \inner{f_1 \cdot (l_q f_2), \psi}_q \cdot \inner{\psi, \abs{F}^2}_q 
        &= \tilde{\lambda}_{f_2}(q) \inner{\sum_{\psi \in \Bc(Y_1)} \inner{f_1 f_2, \psi}_1 \psi, \abs{F}^2}_q \\
        &= \tilde{\lambda}_{f_2}(q) \inner{f_1 f_2, \abs{F}^2}_q \ll_{f_1, f_2} q^{\theta},
    \end{align*}
    where the implied constant depends polynomially on the eigenvalue of $f_i$ using any sup-norm bound, e.g. \cite{iwaniec-sarnak}. 
    We used here the assumption that $f_2$ is not the constant function (see Proposition~\ref{prop:que-level-glh}) and also Remark \ref{rem:norms} to pass between $\inner{\cdot, \cdot}_q$ and $\inner{\cdot, \cdot}_1$ for old forms.

    Next, consider the contribution of the elements $\psi_q$, where we may assume that $\psi$ is not the constant function.
    Recall that
    \begin{displaymath}
        \psi_q = \frac{1}{\sqrt{1 - \tilde{\lambda}_\psi(q)^2}}(l_q\psi - \tilde{\lambda}_\psi(q) \cdot \psi),
    \end{displaymath}
    and that $\tilde{\lambda}_\psi(q) \ll q^\theta$ with $\theta < 0$.
    The latter also implies the bound $(1-\tilde{\lambda}_\psi(q)^2)^{-1/2} \ll 1$.
    For the second term, a treatment as above would be desirable, except the spectral expansion is now twisted by the normalization factor.
    We therefore apply the Watson-Ichino formula.

    Consider the first term in the definition of $\psi_q$, that is $(1-\tilde{\lambda}_\psi(q)^2)^{-1/2} \cdot l_q \psi$.
    Observe that, by self-adjointness of the Hecke operator,
    \begin{displaymath}
        \inner{f_1, l_q(\bar{f_2} \psi))} = \inner{\tilde{T}_q f_1, \bar{f}_2 \psi} = \tilde{\lambda}_{f_1}(q) \cdot \inner{f_1 f_2, \psi}.
    \end{displaymath}
    Additionally, by changing variables using the Fricke involution in the integral defining the inner product, we obtain
    \begin{displaymath}
        \inner{(l_q \psi), \abs{F}^2}_q = \inner{\psi, \abs{F}^2}_q,
    \end{displaymath}
    since $l_q W_q(\psi) = \psi$, as observed in Section \ref{sec:invol}, and $F$ is an eigenfunction of the involution $W_q$.

    By the Watson-Ichino formula, Proposition~\ref{prop:watson-ichino}, we deduce that
    \begin{displaymath}
        \inner{f_1 \cdot (l_qf_2), \psi_q} \ll q^\theta \sqrt{\frac{\Lambda(1/2, f_1 \times f_2 \times \psi)}{\prod_i \Lambda(1, \sym^2 f_i) \Lambda(1, \sym^2 \psi)}}
    \end{displaymath}
    and
    \begin{displaymath}
        \inner{\psi_q, \abs{F}^2} \ll \frac{1}{\sqrt q} \cdot \sqrt \frac{\Lambda(1/2, \psi \times F \times F)}{\Lambda(1, \sym^2 \psi)\Lambda(1, \sym^2 F)^2}.
    \end{displaymath}
    Here, we simply bounded $(1-\tilde{\lambda}_\psi(q)^2)^{-1/2} \ll 1$ and $\tilde{\lambda}_\psi(q), \tilde{\lambda}_{f_1}(q) \ll q^\theta$, using the assumption that $f_1$ is non-constant.
    Looking towards applying Lemma~\ref{lem:stirling}, we compute that
    \begin{equation}\label{eq:pmtpmt1}
        \sum_{\{\pm \}^2} |t_{\psi}+\pm t_{f_1}+\pm t_{f_2}| -2(t_{\psi}+t_{f_1}+t_{f_2}) \geq  2\max(t_\psi-t_{f_1}-t_{f_2},0).
    \end{equation}
    and
    \begin{equation}\label{eq:pmtpmt2}2t_\psi+|2t_F+t_\psi|+|2t_F-t_\psi|-2t_\psi-4t_F =2\max( t_\psi-2t_F,0)\geq 0.
    \end{equation}
    Using also a theorem of Hoffstein--Lockhart, as in the proof of Proposition~\ref{prop:que-level-glh}, we have that
    \begin{multline*}
        \inner{f_1 \cdot (l_qf_2), \psi_q} \inner{\psi_q, \abs{F}^2} \\ \ll q^{\theta - 1/2 + \varepsilon} e^{- \frac{\pi}{2}\max(t_\psi -t_{f_1}-t_{f_2},0)} \sqrt{L(1/2, f_1 \times f_2 \times \psi)L(1/2, \psi \times F \times F)}
    \end{multline*}

    The Generalized Lindelöf Hypothesis implies now that
    \begin{displaymath}
        \sum_{\psi \in \Bc(Y_1)} \inner{f_1 \cdot (l_q f_2), \psi_q }_q \cdot \inner{\psi_q, \abs{F}^2}_q \ll q^{\theta - 1/2 + \varepsilon} \sum_{\psi \in \Bc(Y_1)} e^{- \frac{\pi}{2}\max(t_\psi -t_{f_1}-t_{f_2},0)}.
    \end{displaymath}
    The contribution in the latter sum from $t_\psi<t_{f_1}+t_{f_2}$ is bounded by $O_B((t_{f_1}+t_{f_2})^2)$ by the Weyl law and the tail $t_\psi > t_{f_1}+t_{f_2}$ is uniformly bounded in terms of $\vol(Y_1) \asymp \disc(B)$ by the Weyl law and a dyadic decomposition.
\end{proof}

\subsection{Equidistribution in stages}\label{sec:equid-stages}
We consider now the case where $t_F$ goes to infinity fast enough in terms of $q$.
The theory of quantum chaos suggests that the mass of $F$ should equidistribute along $Y_q$.
Together with the equidistribution of the Hecke correspondence, we expect that proving our conjecture would require weaker hypotheses in this regime.
Indeed, we describe in this section how the Generalized Lindelöf Hypothesis is enough.

\begin{cor}
 \label{cor:bound-on-triple-prod} 
    For $L^2$-normalized Hecke--Maaß cuspidal newforms $F, \varphi$ on $Y_q$ and $f_1,f_2$ Hecke--Maaß cuspforms on $Y_1$, it holds that
    \begin{align*}
       & \inner{f_1 \cdot (l_q f_2), \varphi}_q \cdot \inner{\varphi, \abs{F}^2}_q \\
       & \ll_{f_1, f_2} e^{-\pi/2 \max(t_\varphi-t_{f_1}-t_{f_2} ,0)}(\abs{2t_F+t_\varphi}\cdot \abs{2t_F-t_\varphi} + 1)^{-1/4}  \\ 
       & \cdot \frac{\sqrt{L(1/2,F\times F\times \varphi)L(1/2,f_1\times f_2\times \varphi)}}{q L(1,\sym^2 \varphi) L(1,\sym^2 F)},
    \end{align*}
    with polynomial dependence on $t_{f_1},t_{f_2}$.
\end{cor}
\begin{proof}
    This follows from the Watson--Ichino formula in Proposition~\ref{prop:watson-ichino} together with Lemma~\ref{lem:stirling} and the elementary inequalities (\ref{eq:pmtpmt1}) and (\ref{eq:pmtpmt2}).
\end{proof}

\begin{prop} \label{prop:in-stages}
    Under GLH, it holds that 
    \begin{align*}
        \langle |F|^2,f_1(q.f_2)\rangle_q\ll_{ \varepsilon,f_i} (q^{\theta + \varepsilon}+q^\varepsilon (t_F+1)^{-1/2+\varepsilon}), 
    \end{align*}
    with polynomial dependence $t_{f_1}$ and $t_{f_2}$ (with exponents allowed to depend on $\varepsilon$) and on $\varepsilon>0$.
    
    In particular, for any $\varepsilon>0$, Conjecture~\ref{conj:mixingQUE} holds for $t_F \geq q^\varepsilon$ and $q \rightarrow \infty$.
\end{prop}

\begin{proof}
We start with the trivial estimate
\begin{equation}\label{eq:supnormL2}
    \inner{f_1 \cdot (l_q f_2), \abs{F}^2}_q\ll (t_{f_1}+1)^{100}(t_{f_2}+1)^{100} \norm{F}^2, 
\end{equation}
using sup-norm bounds for $f_1,f_2$ \cite{iwaniec-sarnak}. 
Since $F$ is $L^2$-normalized and the implied constant is allowed to depend polynomially on $t_{f_1},t_{f_2}$, we may assume that $t_{f_1}+t_{f_2}<q^{\varepsilon}$.

We apply the spectral expansion \eqref{eq:spectral-expansion} where, by Lemma~\ref{lemma:old-spectrum-negligible}, we may restrict our attention to the newforms $\varphi \in \Bc_{\mathrm{new}}(Y_q)$.
    Under GLH, Corollary \ref{cor:bound-on-triple-prod} gives the bound
    \begin{multline*}
        \sum_{\varphi \in \Bc_{\mathrm{new}}(Y_q)} \inner{f_1 \cdot (l_q f_2), \varphi}_q \cdot \inner{\varphi, \abs{F}^2}_q \\ 
        \ll  q^{-1 + \varepsilon} \sum_{\varphi \in \Bc_{\mathrm{new}}(Y_q)} (\abs{2t_F+t_\varphi}+1)^{-1/4+\varepsilon}( \abs{2t_F-t_\varphi} + 1)^{-1/4+\varepsilon} e^{-\frac{\pi}{2}\max( t_\varphi-t_{f_1}-t_{f_2},0)}.
    \end{multline*}
    The contribution in the sum from  $t_\varphi>2(t_1+t_2)$ is bounded by 
    \begin{equation*}
      \sum_{\varphi \in \Bc_{\mathrm{new}}(Y_q)} (\abs{2t_F+t_\varphi}+1)^{-1/4+\varepsilon}( \abs{2t_F-t_\varphi} + 1)^{-1/4+\varepsilon} e^{-\frac{\pi}{4} t_\varphi},  
    \end{equation*}
    which is admissible by using a dyadic decomposition and the Weyl law. The contribution  from  $t_\varphi\leq 2(t_{f_1}+t_{f_2})$ is bounded by the Weyl law by
    \begin{equation*}
       \ll (2t_F+1)^{-1/2+\varepsilon} q (t_{f_1}+t_{f_2})^2\ll_{f_1, f_2} q (2t_F+1)^{-1/2+\varepsilon},  
    \end{equation*}
    which is also admissible.
\end{proof}
\begin{rem}
    More generally, if one assumes a subconvexity bound of the shape 
    \begin{equation*}
        L(1/2,F\otimes F\otimes \varphi)\ll q^A (t_F+1)^{1-\delta},
    \end{equation*}
    for $A>0$ and $0<\delta<1$, then Conjecture~\ref{conj:mixingQUE} holds in the regime $t_F\geq q^{\delta^{-1}A+\kappa}$ for any fixed $\kappa>0$.
\end{rem}
\begin{rem}\label{rem:t1t2}
    Note that by the bound (\ref{eq:supnormL2}), since we allow for polynomial dependence on $t_{f_1},t_{f_2}$, for any $\varepsilon>0$ we may throughout assume that $t_{f_1},t_{f_2} \leq q^\varepsilon$, by possibly changing the exponents.   
\end{rem}
\subsection{Reduction to a fractional moment}
Lemma~\ref{lemma:old-spectrum-negligible} reduces the problem to dealing with the terms in the spectral expansion \eqref{eq:spectral-expansion} coming from newforms $\Bc_\mathrm{new}(Y_q)$.
We again apply the explicit Watson--Ichino formula, Proposition~\ref{prop:watson-ichino}, to translate to a fractional moment problem, which we phrase in terms of $L$-functions for newforms on $Y_0(qD)$ using the Jacquet--Langlands correspondence.
This prepares the ground for using the Kuznetsov formula.

More precisely, let 
\begin{equation*}\mathcal{L}(\varphi):= |L(1/2,F\otimes F\otimes \varphi)L(1/2,f_1\otimes f_2\otimes \varphi)|,\end{equation*}
and 
\begin{equation}\label{eq:h(t)}
    h(t):= \left(\cosh \frac{\pi t}{ 2(t_{f_1}+t_{f_2}+3)}\right)^{-1},
\end{equation}
which satisfies that $h(t)>0$ for $t\in \R\cup i[-1/2,1/2]$.
Applying Corollary \ref{cor:bound-on-triple-prod} together with the elementary inequalities,
\begin{equation*}
    \max(t-r,0)> \frac{t}{r+3}-2,\quad e^{-|t|}\leq (\cosh t)^{-1},\quad t,r>0,
\end{equation*}
for tempered $\varphi$, and 
\begin{equation*}
  \left(\cosh \frac{i\pi s}{6}\right)^{-1}=\left(\cos \frac{\pi s}{6}\right)^{-1}\geq 1,\quad s\in \left[-\frac{1}{2},\frac{1}{2}\right],    
\end{equation*}
for non-tempered $\varphi$,
 as well as the Jacquet--Langlands correspondence, we arrive at
\begin{align*}
    &\sum_{\varphi \in \Bc_\mathrm{new}(Y_q)}\langle |F|^2,\varphi\rangle_q \langle \varphi,f_1 \cdot (l_qf_2)\rangle_q\\
    &\ll_{f_1,f_2} \frac{q^{-1}}{L(1,\sym^2 F)}\sum_{\varphi \in \Bc_\mathrm{new}(Y_0(qD))} \frac{h(is_\varphi)}{L(1,\sym^2 \varphi)} \mathcal{L}(\varphi)^{1/2},     
\end{align*}
with polynomial dependence on $t_{f_1},t_{f_2}$.
The main technical result of the paper is the following:
\begin{theorem}\label{thm:mainestimate}
    Assume GRH. Then it holds for any $\varepsilon>0$ that
    \begin{equation*}
        \sum_{\varphi\in \Bc_{\mathrm{new}}(Y_0(qD))}\frac{h(i s_\varphi)}{L(1,\sym^2 \varphi)} \mathcal{L}(\varphi)^{1/2}\ll_{\varepsilon,f_i}\begin{cases}
            q (\log q)^{-3/8+\varepsilon},& f_1\neq f_2\\
            q (\log q)^{-1/4+\varepsilon},& f_1= f_2
        \end{cases},
    \end{equation*}
    where the implied constant depends polynomially on  $t_{f_1},t_{f_2}$. 
\end{theorem}
In view of Weyl's criterion as in Lemma~\ref{lem:Weyl} and standard bounds for the symmetric square $L$-function at $s=1$ under GRH (see Corollary \ref{cor:sym2at1} below) this implies Theorem~\ref{thm:main-intro}. The rest of the paper is devoted to the proof of Theorem~\ref{thm:mainestimate}.  

\section{A random model for fractional moments}\label{sec:heuristics} 
In this section we explain a heuristic for calculating fractional moments of $L$-functions and how to apply it in the setting of Theorem~\ref{thm:mainestimate}. See also \cite[Sec.~5]{Blomer2024} and \cite[Sec.~5.1]{lester-radz}. 

For any reasonable family $\mathcal{F}$ of $L$-functions one should expect the analogue of Selberg's central limit theorem to hold.
This says that the numbers $\{\log |L(1/2,\pi)|: \pi \in \mathcal{F}(X)\}$  are asymptotically normally distributed with mean and variance of size $\asymp \log\log |\mathcal{F}(X)|$, where  $\mathcal{F}(X)=\{\pi \in \mathcal{F}: \mathbf{c}(\pi)\leq X\}$. 
We will now present a random model that explains this and that allows one to calculate the mean and variance. 
We will apply this in the case of the family
$$\mathcal{F}(q)=\{\pi_0\otimes \varphi: \varphi\in \Bc_\mathrm{new}(q),t_\varphi\leq 100\},\quad \pi_0=F\otimes F\boxplus f_1\otimes f_2,$$ 
where $F$ has level $q$ and spectral parameter $t_F\leq q^{100}$. 
This means that 
$L(1/2,\pi_0\otimes \varphi)=L(1/2,F\otimes F\otimes \varphi)L(1/2,f_1\otimes f_2\otimes \varphi)$.

The starting point for the random model is the assumption that we can approximate the central value by a short Dirichlet polynomial.
That is, after Taylor expanding the logarithm, we have
\begin{align*}
    \log L(1/2,f_1\otimes f_2\otimes f_3)=&\sum_{p\leq x} \frac{\lambda_{f_1}(p)\lambda_{f_2}(p)\lambda_{f_3}(p)}{p^{1/2}}\\
    &+\frac{1}{2}\sum_{p^2\leq x} \frac{(\lambda_{f_1}(p^2)-1)(\lambda_{f_2}(p^2)-1)(\lambda_{f_3}(p^2)-1)}{p}+(\text{bounded}),
\end{align*}
where $x=Q^\varepsilon$ is a small power of the analytic conductor $Q$ of $f_1\otimes f_2\otimes f_3 $. Here we are using the Ramanujan conjecture to bound the contribution from higher powers of $p$. 
Assume for the moment that $f_1,f_2,f_3$ are pairwise distinct, with trivial central characters and square-free levels.
We then get by expanding the product that
\begin{align*}
    &\sum_{p^
    2\leq x} \frac{(\lambda_{f_1}(p^2)-1)(\lambda_{f_2}(p^2)-1)(\lambda_{f_3}(p^2)-1)}{p}-\\
    &=-\sum_{p^
    2\leq x} \frac{1}{p}+\sum_{\pi \in \{\sym^2 f_1,\ldots, \sym^2 f_1\otimes \sym^2 f_2\otimes \sym^2 f_3  \}}\pm \sum_{p^
    2\leq x} \frac{\lambda_\pi(p)}{p}.
\end{align*}
The terms in the last sum are all heuristically equal to the logarithm of a holomorphic $L$-function at $s=1$ up to a bounded error.
This is true under GRH by Corollary \ref{cor:boundsL(1)}, using here the assumptions on $f_1,f_2,f_3$, as well as appropriate instances of functoriality (see Section \ref{sec:l-functions}). 
By Mertens' second theorem we arrive at
\begin{align*}
    \log L(1/2,f_1\otimes f_2\otimes f_3)= -\frac{1}{2}\log \log Q + \sum_{p\leq x} \frac{\lambda_{f_1}(p)\lambda_{f_2}(p)\lambda_{f_3}(p)}{p^{1/2}}\\
    + \log \left\{\text{holomorphic $L$-functions at $s=1$}\right\}+(\text{bounded}).
\end{align*}
Assuming GRH and the Ramanujan conjecture, we know that the contribution from the $L$-functions at $s=1$ is of size $\log\log\log Q$ and is thus negligible. 
When $f_1=f_2$ a similar calculation yields the same expression but with $-\log\log Q$ in place of $-\frac{1}{2}\log \log Q$. This way we arrive at the approximation
\begin{equation}\label{eq:randommodel}
  \log L(1/2,\pi_0\otimes \varphi) \approx  \mu_q+\sum_{p\leq x} \frac{\left( \lambda_F(p)^2+\lambda_{f_1}(p)\lambda_{f_2}(p) \right) \lambda_{\varphi}(p)}{p^{1/2}},\quad x=q^\varepsilon
\end{equation}
where $\mu_q=-2\log\log q $ or $\mu_q = -\frac{3}{2}\log\log q$, according to whether $f_1=f_2$ or not.

Now, for each prime $p\leq x$, we can think of 
$$\varphi\mapsto \frac{\lambda_{\varphi}(p)}{p^{1/2}},$$ 
as corresponding to independent random variables $X_p$ with mean zero and variance $p^{-1}$, since $|\lambda_{\varphi}(p)|^2$ is $1$ on average by an application of the Kuznetsov trace formula. 
In this case, we would expect by the central limit theorem that, for a sequence $b(p)$ of real numbers, $\sum_{p\leq x} b(p)X_p$ should be well approximated by a normal distribution with mean zero and variance $\sum_{p\leq x} b(p)^2p^{-1}$. 
Applying this to (\ref{eq:randommodel}) and using the Hecke relations as in Lemma~\ref{lem:variance-coeffs}, we arrive at the following value for the variance:
\begin{align}\label{eq:varrandommodel}
    \sum_{p\leq x} \frac{(\lambda_F(p)^2+\lambda_{f_1}(p)\lambda_{f_2}(p))^2}{p}&\approx  \begin{cases}
       3\log \log q,& f_1\neq f_2\\
       6\log \log q,& f_1=f_2.
    \end{cases}\\
 \nonumber   &+\log \left\{\text{holomorphic $L$-functions at $s=1$}\right\}.
\end{align}
Combining all of the above we conclude that the logarithms of the family of $L$-functions appearing in our fractional moment is well-modeled by a Gaußian with mean $-\frac{3}{2}\log\log q$, respectively $-2\log \log q$, and variance $3\log\log q$, respectively $6\log\log q$, according to whether $f_1=f_2$ or not. 

Finally, let $X\sim \mathcal{N}(\mu,\sigma^2)$ be a Gaußian random variable. Then a quick calculation shows that for any $t>0$,
\begin{equation}\label{eq:expnormal}\mathbb{E}(e^{tX})=e^{t\mu+\frac{1}{2}t^2\sigma^2}.\end{equation} 
Applying this to our Gaußian model above we would thus expect that
\begin{align*}
   \frac{1}{|F(q)|} \sum_{\varphi}\sqrt{L(1/2,F\otimes F\otimes \varphi)L(1/2,f_1\otimes f_2\otimes \varphi)}\approx \begin{cases}
       (\log q)^{-3/8},& f_1\neq f_2\\
       (\log q)^{-1/4},& f_1= f_2.
   \end{cases} 
\end{align*}
We will see that, under GRH, we can recover the upper bound up to a factor $(\log q)^\varepsilon$.

\section{Analytic and automorphic tools}

Write $D_B$ for the reduced discriminant $\disc(B)$ of the indefinite quaternion algebra $B$. 
We recall that the theory of Jacquet-Langlands gives a one-to-one correspondence between Hecke--Maaß forms on $Y_1$ and Hecke--Maaß newforms on $Y_0(D_B)$, the quotient of $\Hb$ by the standard Hecke congruence subgroup $\Gamma_0(D_B)$ (see \cite[Thm~8.18]{bergeron} or \cite{Strombergson01}).
Similarly, Hecke--Maaß newforms on $Y_q$ are in one-to-one correspondence with Hecke--Maaß newforms on $Y_0(qD_B)$.
Note that both $D_B$ and $qD_B$ are square-free. 
As such, we may import much of the analytic theory of automorphic forms for $\Gamma_0(N)$ with $N$ square-free.

\subsection{The Bruggeman--Kuznetsov formula} 
Let $h: \R\cup i(-1/2,1/2)\rightarrow \C$ be a test function that can be extended to the strip $|\Im(t)|<1/2+\delta$ for some $\delta>0$ satisfying the following properties:
\begin{enumerate}
    \item $h$ is even and holomorphic, 
    \item $h(t)\ll (1+|t|)^{-2-\delta}$ for $|\Im(t)|<1/2+\delta$.
\end{enumerate}
In particular, this is satisfied by $t\mapsto  \cosh(\pi t/r)^{-1}$ for $r\geq 2$. 
For a positive integer $N$, let $\Bc_\mathrm{new}(Y_0(N))$ be an orthonormal basis of newforms of level $N$.   
\begin{theorem}[{\cite[Theorem 9.3]{iwaniec-spectral}}] \label{thm:kuzn} 
    Let $N>0$ be a square-free integer. 
    Then, for $h$ as above and $m,n>0$ coprime to $N$, it holds that
    \begin{align*}
        &\sum_{LM=N}\sum_{\varphi\in \Bc_\mathrm{new}(Y_0(M))}\frac{h(i s_\varphi)}{L(1,\sym^2 \varphi)}\mathcal{W}_\varphi(L)  \lambda_\varphi(m)\lambda_\varphi(n)\\
        &+\frac{1}{2\pi}\int_\R  \frac{h(t)}{|\zeta(1+2it)|^2} \mathcal{W}_t(N)\lambda_t(m)\lambda_t(n) dt
        \\
        &=N\delta_{m=n} h^0 +\sum_{c\geq 1} \frac{S(m,n;Nc)}{c} h^+\left( \frac{4\pi \sqrt{mn}}{Nc}\right),
    \end{align*}
    where $\mathcal{W}_\varphi, \mathcal{W}_t: \mathbb{N}\rightarrow \R_{>0}$ are completely multiplicative positive functions, uniquely determined by
  \begin{align}
        \label{eq:Wvarphi}\mathcal{W}_\varphi(p):=\frac{p}{2}\left(1+\frac{\lambda_\varphi(p)^2}{p(1-(\lambda_\varphi(p^2)-1)p^{-1}+p^{-2})}\right),\\
        \label{eq:Wt}\mathcal{W}_t(p):=\frac{p}{2}\left(1+\frac{\lambda_t(p)^2}{p(1-p^{-1+2it})(1-p^{-1-2it})}\right),
    \end{align}
    and 
    \begin{align*}
        h^0:=\frac{1}{\pi}\int_\R h(t)\,  t\tanh \pi t\, dt,\\
        h^+(x):= \int_\R h(t)\frac{2it\,  J_{2it}(x)}{\cosh\pi t}\, dt
    \end{align*}
\end{theorem}
\begin{proof}
This follows from the explicit formulas in \cite[Lemma A.9]{HumKhan20}, which simplifies significantly since $(mn,N)=1$, and that by (\ref{eq:localsym2}) the local symmetric square $L$-function at unramified finite places is given by
\begin{align*}
L_p(s,\sym^2\varphi )^{-1}&=(1-\alpha_{\varphi,p}^{2}p^{-s})(1-\alpha_{\varphi,p}^{-2} p^{-s})(1- p^{-s})\\
&=(1-(\lambda_\varphi(p^2)-1)p^{-s}+p^{-2s})(1-p^{-s}),\end{align*}
where $\alpha_{\varphi,p}^{\pm 1}$ denotes the Satake parameters of $\varphi$ at $p$. The fact that $\mathcal{W}_\varphi(p)$ is non-negative follows from the trivial bound $|\lambda_\varphi(p^2)|\leq p$. A similar argument applies to the continuous spectrum. 
\end{proof}
We will now calculate the transforms for the relevant test functions.
\begin{lemma}\label{lem:h0} 
  Let $r\geq 1$ and $h(t)=(\cosh (\pi t/r))^{-1}$. Then it holds that
    $$h^0\ll r^2.$$   
\end{lemma}
\begin{proof}
    By the trivial inequalities $|\tanh t|\leq 1$ and $(\cosh t)^{-1}\leq 2 e^{-|t|}$,
    \begin{align*}
      h^0=\frac{1}{\pi}\int_\R h(t)\,  t\tanh \pi t\, dt  \ll \int_{\R} |t| e^{-\pi |t|/r} dt\ll r^2, 
    \end{align*}
    as wanted.
\end{proof}
\begin{lemma}\label{lem:h+}
    Let $r\geq 3$ and $h(t)=(\cosh (\pi t/r))^{-1}$. Then it holds that
    $$h^+(x)= \frac{x}{\cos \frac{\pi}{2r}}+O(r^{1/2} x^2),\quad 0<x<2.$$
\end{lemma}
\begin{proof}
By the Taylor expansion of the $J$-Bessel function \cite[(B.28)]{iwaniec-spectral}, it holds that
\begin{align*}
    J_{2it}(x)= \sum_{j=0}^\infty \frac{(-1)^j\left(\frac{x}{2}\right)^{j+2it}}{j!\, \Gamma(j+1+2it)} =\frac{\left(\frac{x}{2}\right)^{2it}}{\Gamma(1+2it)}+O\left(x^2\sum_{j=1}^\infty \frac{1}{j!\, |\Gamma(j+1+2it)|} \right),
\end{align*}
for $x < 2$.
By the functional equation for the $\Gamma$-function and Stirling's formula we have for $j\geq1$ that
$$|\Gamma(j+1+2it)|=|(j+2it)\cdots (1+2it)| |\Gamma(1+2it)|\gg (|t|+1)^{3/2} e^{-\pi |t|}.$$
Inserting this into the definition of $h^+$ and using that $(\cosh \pi t)^{-1}(\cosh \pi t/r)^{-1}\ll e^{-\pi(1+1/r)|t|}$ we arrive at 
\begin{align*}
    h^+(x)&=-i\int_{(0)} \frac{2z\left(\frac{x}{2}\right)^{2z}}{\Gamma(1+2z)(\cos \pi z)(\cos \pi z/r)}\, dz+O\left(x^2\int_{\R} \frac{|t|}{e^{-\pi |t|}(|t|+1)^{3/2}e^{\pi (1+1/r)|t|}}\, dt  \right)\\
    &=-i\int_{(0)} \frac{2z\left(\frac{x}{2}\right)^{2z}}{\Gamma(1+2z)(\cos \pi z)(\cos \pi z/r)}\, dz+O\left( r^{1/2} x^2 \right). 
\end{align*}
Now we shift the contour to $\Re z=1$ picking up the residue at $z=1/2$ (using here the rapid decay as $\abs{\Im z} \rightarrow \infty$ and that $r\geq 3$ to ensure that $z=1/2$ is the unique pole). 
Note that, by the Taylor expansion of the cosine function around $z=\pi/2$, we have 
$$(\cos \pi z)^{-1}=\frac{1}{\pi}(z-1/2)^{-1}(1+O((z-1/2)^2))^{-1}=\frac{1}{\pi}(z-1/2)^{-1}+O((z-1/2)). $$
Thus we conclude that
\begin{align*}
    \mathop{\mathrm{res}}_{z=1/2} \frac{2z\left(\frac{x}{2}\right)^{2z}}{\Gamma(1+2z)(\cos \pi z)(\cos \pi z/r)}
    &= \frac{1}{\pi} \left(\frac{2z\left(\frac{x}{2}\right)^{2z}}{\Gamma(1+2z)\cos \pi z/r}\right) \Biggr\rvert_{z=1/2}\\
    &= \frac{x}{2\pi \cos\frac{\pi}{2r}}.
\end{align*}
By considerations as above, the contribution from the vertical line $\Re z=1$ is seen to be $O(x^2)$. 
Thus, we get the desired result by the residue theorem.
\end{proof}
\begin{cor}\label{cor:kuzn}
 Let $N>0$ be a square-free integer. Let $h(t)=(\cosh (\pi t/r))^{-1}$, $r\geq 3$.
 Let $m,n\geq 1$ be integers coprime to $N$ such that $mn< N^2/4\pi^2$. 
 Then it holds that 
    \begin{align*}
        &\sum_{LM=N}\sum_{\varphi\in \Bc_\mathrm{new}(Y_0(M))}\frac{h(i s_\varphi)}{L(1,\sym^2 \varphi)}\mathcal{W}_\varphi(L)  \lambda_\varphi(m)\lambda_\varphi(n)\\
        &+\frac{1}{4\pi}\int_\R  \frac{h(t)}{|\zeta(1+2it)|^2} \mathcal{W}_t(N)\lambda_t(m)\lambda_t(n) dt
        \\
        &=N\delta_{m=n} h^0 +O\left(r^{1/2}  \frac{\sqrt{mn}}{\sqrt{N}}\right).
    \end{align*}
\end{cor}
\begin{proof}
By combining Weil's bound (using here that $(mn,N)=1)$ with Theorem~\ref{thm:kuzn} and Lemma~\ref{lem:h+} (note that indeed $0<\frac{4\pi \sqrt{mn}}{N}<2$), we conclude that the left-hand side of the above equation equals
\begin{align*}
    N\delta_{m=n} h^0 +O\left(r^{1/2}  \sum_{c\geq 1} \frac{\sqrt{Nc}}{c} \frac{\sqrt{mn}}{Nc} \right).
\end{align*}
This yields the desired formula.
\end{proof}

\subsection{$L$-functions at $s=1$}
In this section we will collect some useful facts about bounds for $L$-function at $s=1$ under GRH. We note that, without assuming the Ramanujan conjecture, this is a very delicate question. 
\begin{lemma}\label{lem:L(1)=sum_p}
 Let $L(s,\pi)$ be a holomorphic $L$-function of degree $d$ in the extended Selberg class (not necessarily primitive) and assume GRH for $L(s, \pi)$. 
 Assume also that the Satake parameters satisfy $\abs{\alpha_\pi(p,j)} \leq p^{\delta}$ for some $0\leq \delta<1/2$. Then for every $\varepsilon>0$ and $x\geq 2$ we have
\begin{equation*}
    \sum_{p<x}\frac{\lambda_\pi(p)}{p}=\log L(1,\pi)+O_{\delta}\left(\log \mathbf{c}(\pi)\log \log \mathbf{c}(\pi)\frac{(\log x)^3}{x^{1/2}}\right).
\end{equation*}
In particular for $x\geq (\log \mathbf{c}(\pi))^{2+\varepsilon}$ it holds that
\begin{equation*}
    \label{eq:L(1)=sumO(1)}\sum_{p<x}\frac{\lambda_\pi(p)}{p}=\log L(1,\pi)+O_{\delta,\varepsilon}(1).
\end{equation*}
\end{lemma}
\begin{proof}
    This follows directly from the statement and the proof of \cite[Lemma~5]{Blomer2024}.
\end{proof}

\begin{cor}\label{cor:boundsL(1)}
    Let $L(s,\pi_1),L(s,\pi_2)$ be two holomorphic $L$-functions in the extended Selberg class (not necessarily primitive). 
    Assume GRH for $L(s,\pi_i)$ and that the Satake parameters of $L(s,\pi_i)$ satisfy the bound in the previous lemma.
    Let $a,b,\alpha\geq 0$ and assume that
    \begin{equation*}
        a\lambda_{\pi_1}(p)-b\lambda_{\pi_2}(p)\leq\alpha, 
    \end{equation*}
    for all primes $p$, where $\lambda_{\pi_i}(n)$ denotes the Dirichlet coefficients of $L(s,\pi_i)$. 
    Then, it holds that
    \begin{align*}
    L(1,\pi_1)^a L(1,\pi_2)^{-b} \ll_{a,b,\alpha} (\log\log (\mathbf{c}(\pi_1)\mathbf{c}(\pi_2)))^{\alpha}
    \end{align*}
    Similarly, if 
    \begin{equation*}
        a\lambda_{\pi_1}(p)-b\lambda_{\pi_2}(p) \geq - \alpha, 
    \end{equation*}
    then
    \begin{align*}
    L(1,\pi_1)^a L(1,\pi_2)^{-b} \gg_{a,b,\alpha} (\log\log (\mathbf{c}(\pi_1)\mathbf{c}(\pi_2)))^{-\alpha}.
    \end{align*}
\end{cor}
\begin{proof}
    This follows from the previous lemma applied to $L(s,\pi_1)$ and $L(s,\pi_2)$ with $x=(\log \mathbf{c}(\pi_1)\mathbf{c}(\pi_2))^3$.
\end{proof}

\begin{cor}\label{cor:sym2at1}
  Let $\pi$ be an automorphic representation for $\PGL_2/\Q$ and assume the conditions of Lemma \ref{lem:L(1)=sum_p} for $ L(s,\sym^2 \pi)$, including GRH. Then it holds that 
  \begin{equation*}
      L(1,\sym^2 \pi)\gg (\log \log \mathbf{c}(\pi))^{-1}
  \end{equation*}
\end{cor}
\begin{proof}
    This follows from Corollary \ref{cor:boundsL(1)} since, using the Hecke relations,
    $$\lambda_\pi(p^2)=\lambda_\pi(p)^2-1\geq -1.$$
\end{proof}
\begin{cor}\label{cor:boundsats=1}
    Let $L(s,\pi)$ satisfy the conditions of Lemma~\ref{lem:L(1)=sum_p}, including GRH. Then for every $\varepsilon>0$ 
    \begin{equation*}
\exp\left(- (\log \mathbf{c}(\pi))^{2\delta +\varepsilon}\right)   \ll_{\varepsilon, \delta}     L(1,\pi)\ll_{\varepsilon, \delta} \exp\left( (\log \mathbf{c}(\pi))^{2\delta +\varepsilon}\right).  
    \end{equation*}
\end{cor}
\begin{proof}
    This follows by Lemma~\ref{lem:L(1)=sum_p} with $x=(\log \mathbf{c}(\pi))^{2 +\varepsilon}$.
\end{proof}

\subsection{$L$-functions at the central point}
In this section we will state the key upper bound for $L$-functions at the central point $s=1/2$ under GRH. This builds on an inequality firstly proved by Soundararajan \cite{Sound09} for the Riemann $\zeta$-function and then extended by Chandee \cite[Theorem 2.1]{Chandee09} to general $L$-functions.
Our treatment becomes more complicated compared to \cite[Sec.~6.3]{Blomer2024}, due to the appearance of higher degree $L$-functions and the need for more delicate arguments in order to avoid GRC.
See also Remark \ref{rem:largerxGRHbound} below.
\begin{prop}\label{prop:GRHbound}
    Let $S$ be a set of primes and let $\pi_1,\pi_2,\pi_3$ be automorphic representations of $\PGL(2)/\Q$. 
    Assume GRH for $L(s,\pi_1\otimes \pi_2\otimes \pi_3)$. 
    Let $0<\varepsilon<1$ and 
    \begin{equation}\label{eq:lowerboundx}
        x\geq \begin{cases}
            4,& \pi_1\not\cong \pi_2\\
            (\log \mathbf{c}(\pi_1\otimes \pi_2)) ^{9/\varepsilon},& \pi_1\cong \pi_2
        \end{cases}
    \end{equation}
    Then there exists a constant $B>0$ such that 
    \begin{align*}
       \log L(1/2,\pi_1\otimes \pi_2\otimes \pi_3) &\leq (1-\varepsilon)\mu_{\pi_1,\pi_2}(x) + \mathcal{S}_{1/2}(\pi_1\otimes \pi_2\otimes \pi_3;x ) \\
       &+ \sum_{n=2}^5 \mathcal{S}_{n/2}(\pi_1\otimes \pi_2\otimes \pi_3;x ) \\
       & +10\frac{\log \mathbf{c}(\pi_1\otimes \pi_2\otimes \pi_3 )}{\log x} \\
       &+O_S(\log \log \log \mathbf{c}(\pi_1\otimes \pi_2\otimes \pi_3)),
    \end{align*}
    where the first term is given by
    \begin{equation}\label{eq:defmupi_i}
        \mu_{\pi_1,\pi_2}(x):=
        \begin{cases}
        -\frac{1}{2}\log \log x+B (\log \mathbf{c}(\pi_1\otimes \pi_2))^{15/16},& \pi_1\neq \pi_2,\\
        -\log\log x-\frac{1}{2}\log L(1,\sym^4 \pi)+\frac{1}{2}\log L(1,\sym^2 \pi),& \pi_1=\pi_2=\pi,
        \end{cases}
    \end{equation}
    the second term is
    \begin{equation}\label{eq:pterms}
        \mathcal{S}_{1/2}(\pi_1\otimes \pi_2\otimes \pi_3;x ) := \sum_{p\leq x, p\notin S}\frac{\lambda_{\pi_1}(p)\lambda_{\pi_2}(p)\lambda_{\pi_3}(p)}{p^{1/2+1/\log x}} \cdot \frac{\log (x/p)}{\log x},
    \end{equation}
    and, for $2\leq n\leq 5$,
    \begin{multline*}
        \mathcal{S}_{n/2}(\pi_1\otimes \pi_2\otimes \pi_3;x )\\
        :=\frac{1}{n}\sum_{p^n \leq x, p\notin S}\frac{(\sum_{m=1}^nc(m,n)\lambda_{\pi_3}(p^m)) \prod_{i=1}^2 (\sum_{m=0}^nc(m,n)\lambda_{\pi_i}(p^m))}{p^{n/2+n/\log x}}\frac{\log (x/p^n)}{\log x},
    \end{multline*}
    for certain constants $c(m,n)\in \Z$. 
\end{prop}

\begin{proof}
The starting point is the following key inequality due to Chandee \cite[Theorem 2.1]{Chandee09}\footnote{We note that there is a typo in \emph{loc. cit.}: there is a $\log$ missing in front of $\abs{L(f, 1/2)}$.}, extending ideas of Soundararajan \cite{Sound09}, which in view of (\ref{eq:tripleunram})-(\ref{eq:tripleram3}) reads
\begin{multline} \label{eq:chandee-ineq}
    \log \abs{L(1/2,\pi_1\otimes \pi_2\otimes \pi_3)} \leq \sum_{\substack{p^n\leq x \\ n \geq 0}} \frac{\prod_{i=1}^3(\alpha_{\pi_i}(p,1)^n+\alpha_{\pi_i}(p,2)^n)}{np^{n/2+n/\log x}}\frac{\log (x/p^n)}{\log x} \\
    +\frac{\log \mathbf{c}(\pi_1\otimes \pi_2\otimes \pi_3)}{\log x},\quad \text{for }x \gg 1.
\end{multline}
By the absolute convergence of the Euler product (\ref{eq:Eulerprod}) we conclude that the contribution from $p\in S$ is $O_S(1)$. 
Since $|\alpha_{\pi_i}(p,j)|\leq p^{7/64}$, as noted in \eqref{eq:KimSarnak}, the terms with $n\geq 6$ are absolutely convergent\footnote{Under the Ramanujan conjecture, the terms with $n \geq 3$ would also be absolutely convergent.} and contribute $O(1)$.

For $3\leq n\leq 5$, using the Hecke relations (see Section \ref{sec:hecke-rel}), we have for $p$ unramified
\begin{align}\label{eq:prod123}
  \prod_{i=1}^3(\alpha_{\pi_i}(p,1)^n+\alpha_{\pi_i}(p,2)^n)= \prod_{i=1}^3\left(\sum_{m=0}^n c(m,n)\lambda_{\pi_i}(p^m)\right),
\end{align}
for certain constants $c(m,n)\in \Z$. 
We now claim that the term $m = 0$ in the factor $i = 3$ is negligible.
Indeed, recall from \eqref{eq:hecke-rel-third-pow} that for $n=3$ we have $c(0,3)=0$.
For $n\geq 4$, the contribution in \eqref{eq:chandee-ineq} from the factor with $i=3$ and $m=0$ in \eqref{eq:prod123} is simply bounded, using the Kim--Sarnak bound \eqref{eq:KimSarnak} factor by factor, by
$\ll p^{n \cdot 7/64 + n \cdot 7/64} = p^{7n/32}$. 
Thus, the total contribution from these terms is bounded by
$$\ll_n \sum_{p^n\leq x} \frac{1}{p^{n/2-7n/32}},$$
which is absolutely convergent for $n\geq 4$. 
This proves that the contribution form $3\leq n\leq 5$ is indeed bounded by $$\sum_{n=3}^5\mathcal{S}_{n/2}(\pi_1\otimes\pi_2\otimes\pi_3;x)+O(1).$$

Similarly, we now consider the terms in \eqref{eq:chandee-ineq} where $n = 2$.
Here, as in Section \ref{sec:hecke-rel},
\begin{displaymath}
    \alpha_{\pi_i}(p,1)^2 + \alpha_{\pi_i}(p,2)^2 = \lambda_{\pi_i}(p)^2 - 2 = \lambda_{\pi_i}(p^2) - 1.
\end{displaymath}
To isolate the term $\mathcal{S}_{1}(\pi_1\otimes\pi_2\otimes\pi_3;x)$, we again consider the contribution of the term $m = 0$ in the third factor of \eqref{eq:prod123}.
By the previous computation, this is
\begin{equation*}
    -\frac{1}{2}\sum_{p^2\leq x,p\notin S}\frac{(\lambda_{\pi_1}(p^2)-1)(\lambda_{\pi_2}(p^2)-1)}{p^{1+2/\log x}} \frac{\log (x/p^2)}{\log x}.
\end{equation*}
If $\pi_1=\pi_2=\pi$, then positivity and the Hecke relations imply that the previous display is
\begin{align*}
    &\leq -\frac{1}{2}\sum_{p^2\leq x^{\varepsilon/2},p\notin S}\frac{(\lambda_{\pi}(p^2)-1)^2}{p^{1+2/\log x}} \frac{\log (x/p^2)}{\log x}\\
    &\leq -\frac{1}{2}\sum_{p^2\leq x^{\varepsilon/2},p\notin S}\frac{(\lambda_{\pi}(p^2)-1)^2}{p} e^{-\varepsilon/2}\left(1-\frac{\varepsilon}{2}\right) \\
    &\leq -\frac{1-\varepsilon}{2}\sum_{p\leq x^{\varepsilon/4},p\notin S}\frac{\lambda_{\pi}(p^4)-\lambda_{\pi}(p^2)+2}{p} 
\end{align*}
Here we use the elementary bound $e^{-\varepsilon/2} \geq 1- \varepsilon/2$.
If $x\geq (\log \mathbf{c}(\pi\otimes \pi))^{9/\varepsilon}$, the last expression in the display above is equal to
\begin{displaymath}
    (1-\varepsilon)\mu_{\pi,\pi}(x)+O_\varepsilon(1),
\end{displaymath}
which we deduce using the identities \eqref{eq:sym2sym4}, the holomorphicity from Lemma~\ref{lem:sym2sym4} and finally Lemma~\ref{lem:L(1)=sum_p} to relate the sum over primes to the $L$-values at $s=1$ (which applies by the assumption (\ref{eq:lowerboundx}) on $x$), as well as the estimate $\log\log x^\varepsilon=\log\log x+O_{\varepsilon}(1)$. 

If $\pi_1\neq \pi_2$ we write the contribution as 
\begin{multline*}
    -\frac{1}{2}\sum_{p\leq \sqrt{x},p\notin S}\frac{\lambda_{\sym^2\pi_1\otimes \sym^2\pi_2}(p)-\lambda_{\sym^2\pi_1}(p)-\lambda_{\sym^2\pi_2}(p)}{p^{1+2/\log x}} \frac{\log (x/p^2)}{\log x} \\
    -\frac{1}{2}\sum_{p\leq \sqrt{x},p\notin S}\frac{1}{p^{1+2/\log x}} \frac{\log (x/p^2)}{\log x}
\end{multline*}
For $\pi \in \{\sym^2\pi_1, \sym^2\pi_2, \sym^2\pi_1\otimes \sym^2\pi_2\}$, we have by partial summation that
\begin{align*}
   &\sum_{p\leq \sqrt{x},p\notin S} \frac{\lambda_\pi(p)}{p^{1+2/\log x}} \frac{\log (x/p^2)}{\log x}\\
   &= \int_1^{\sqrt{x}} \left(\sum_{p\leq t,p\notin S} \frac{\lambda_\pi(p)}{p} \right) \left(\frac{2 \log (x/t^2)}{t^{1+2/\log x}(\log x)^2}+\frac{2 }{t^{1+2/\log x}\log x}\right) \, dt\\
   &\ll \int_2^{\sqrt{x}} \left(\sum_{p\leq t,p\notin S} \frac{\lambda_\pi(p)}{p} \right) \frac{1}{t\log x}\, dt.
\end{align*}
We split the integral into two ranges.
First, for any $\eta < 1$ to be chosen later, note that we have the bound
\begin{displaymath}
    \int_{2}^{\log(\mathbf{c}(\pi))^\eta} \left(\sum_{p\leq t,p\notin S} \frac{\lambda_\pi(p)}{p} \right) \frac{1}{t\log x}\, dt \ll \log(\mathbf{c}(\pi))^\eta,
\end{displaymath}
since $|\lambda_\pi(p)| \leq p$ by absolute convergence of the Euler product (\ref{eq:Eulerprod}).
Next, for the complementary range, we apply Lemma~\ref{lem:L(1)=sum_p} and write
\begin{multline*}
    \int_{\log(\mathbf{c}(\pi))^\eta}^{\sqrt{x}} \left(\sum_{p\leq t,p\notin S} \frac{\lambda_\pi(p)}{p} \right) \frac{1}{t\log x}\, dt\\
    \ll \int_{\log(\mathbf{c}(\pi))^\eta}^{\sqrt{x}} \left( |\log L(1,\pi)|+O((\log t)^3t^{-1/2}\log \mathbf{c}(\pi)\log \log \mathbf{c}(\pi))\right) \frac{1}{t\log x}\, dt
\end{multline*}
For the second term in the integral, we bound $\log \mathbf{c}(\pi) \log \log \mathbf{c}(\pi)/ t^{1/4}$ as $O_\varepsilon (\log (\mathbf{c}(\pi))^{1-\eta/4 + \varepsilon})$, and observe that the remaining integral of $\log(t)^3 t^{-1-1/4}$ is uniformly bounded.
For the first term, we apply Corollary \ref{cor:boundsats=1}.
Here, we note that the Satake parameters of any $\pi$ in the set above are bounded by $(p^{7/64})^4 = p^{7/16}$ by \eqref{eq:KimSarnak} (see the discussion on Satake parameters after \eqref{eq:sym2sym4}).
Thus, $\log L(1,\pi) \ll \log (\mathbf{c}(\pi))^{14/16 + \varepsilon}$ for any $\varepsilon > 0$.
All in all, we may choose for example $\eta = 4/5$ and we obtain that
\begin{displaymath}
    \sum_{p\leq \sqrt{x},p\notin S} \frac{\lambda_\pi(p)}{p^{1+2/\log x}} \frac{\log (x/p^2)}{\log x} \ll \log (\mathbf{c}(\pi))^{15/16}.  
\end{displaymath}

We finish the proof by observing that
\begin{align*}
    -\frac{1}{2}\sum_{p\leq \sqrt{x},p \notin S}\frac{1}{p^{1+2/\log x}} \frac{\log (x/p^2)}{\log x}&=-\frac{1}{2}\sum_{p\leq \sqrt{x}}\frac{1}{p} \left(1+O\left(\frac{\log (p)}{\log x}\right)\right)+O_S(1) \\
    &= -\frac{1}{2}\log \log x+O_S(1).
\end{align*}
by Mertens' first and second theorem.
\end{proof}

For our applications, we will apply the above proposition with $\pi_1,\pi_2,\pi_3$ automorphic representations of $\PGL(2)$ with conductors dividing $qD$ with $D$ square-free (and fixed) and $q$ a prime varying. We will pick $S=\{p:p| D\}$ and $x<q$ so that all primes appearing in the sum defining $\mathcal{S}_{n/2}(\pi_1\otimes \pi_2\otimes \pi_3;x)$ are unramified for $\pi_1,\pi_2,\pi_3$.

\begin{rem}\label{rem:largerxGRHbound}
    We note two differences between Proposition~\ref{prop:GRHbound} and \cite[Cor.\ 7]{Blomer2024} (see also \cite[eq.\ (5.6)]{lester-radz}): firstly, when $\pi_1\cong \pi_2$ we need that $x$ is bigger than a large power of $\log q$ (depending on $\varepsilon$), and secondly, we get an upper bound in terms of $(1-\varepsilon)\mu_{\pi_1,\pi_2}(x)$ instead of $\mu_{\pi_1,\pi_2}(x)$. The cause of these differences is that the proof of \cite[Cor.\ 7]{Blomer2024} relies on the inequality $(\lambda(p^2)-1)(1-\eta_E(p))\geq -2$ where $\lambda(n)$ denotes the Hecke eigenvalues of an automorphic representation of $\PGL(2)/\Q$. In our setting the corresponding term is $-(\lambda_{\pi_1}(p^2)-1)(\lambda_{\pi_2}(p^2)-1)$, which apriori, without assuming GRC, can be arbitrarily negative. Thankfully, in the case $\pi_1\not\cong\pi_2$ (where we do not know how to bound the above term  neither from below nor above) we can allow for a polynomial dependence on the conductors of $\pi_1,\pi_2$. Whereas when $\pi_1\cong\pi_2$, the term is always negative which we can use to our advantage when $x$ is bigger than a sufficiently large power of $\log q$ at the cost of replacing the saving of $\mu_{\pi_1,\pi_2}(x)$ by $(1-\varepsilon)\mu_{\pi_1,\pi_2}(x)$. 
\end{rem}

\section{High moments of short Dirichlet polynomials}
In this, section we use the Kuznetsov formula to bound high moments of short averages of Dirichlet polynomials over (essentially) the spectral family $\{\varphi\in \Bc_\mathrm{new}(Y_q): t_\varphi \leq 100\}$. This will be in accordance with random model discussed in Section \ref{sec:heuristics}, according to which $\varphi\mapsto \lambda_{\varphi}(p)/p^{1/2}$ behave like independent random variables $X_p$ with mean zero and variance $p^{-1}$. In this case, we would expect by the central limit theorem for a sequence $b(p)$ of real numbers that $\sum_{p\leq x} b(p)X_p$ should be well approximated by a normal distribution with mean zero and variance $\sum_{p\leq x} b(p)^2p^{-1}$ so that 
\begin{equation*}
 \mathbb{E}\left(\left(  \sum_{p\leq x} b(p)X_p \right)^{2k}\right)= \frac{(2k)!}{2^k k!}\left(\sum_{p\leq x} \frac{b(p)^2}{p}\right)^k.  
\end{equation*}
The following result shows that, for $x$ sufficiently small, we can actually make this heuristic precise if we settle for an upper bound. As in the previous section, put 
\begin{equation*}
    h(t)= \left(\cosh \frac{\pi t}{2(t_{f_1}+t_{f_2}+3)}\right)^{-1}.
\end{equation*} 
Note that $h(t)>0$ for $t \in \R \cup i[-1/2,1/2]$.

\begin{prop}[High moment estimate]\label{prop:key}
    Let $2\leq x$ and let $k$ be a positive integer such that $x^{k}<qD/2\pi$. Let $b(p)$ be a sequence of real numbers indexed by primes. Then it holds that
    \begin{multline}
    \label{eq:highmoment}
        \sum_{\varphi\in \Bc_{\mathrm{new}}(Y_0(qD))} \frac{h(is_\varphi)}{L(1,\sym^2 \varphi)} \left(\sum_{p<x, p\nmid Dq} \frac{b(p)\lambda_\varphi(p)}{p^{1/2}}\right)^{2k} \\
        \ll_{f_1,f_2} q\frac{(2k)!}{k!} \left(\frac{1}{2}\sum_{p<x, p\nmid Dq} \frac{b(p)^2}{p}\right)^{k}+q^{-1/2}  \left(2\sum_{p<x,p\nmid Dq}|b(p)|\right)^{2k},
    \end{multline}
    with polynomial dependence on $t_{f_1},t_{f_2}$.
\end{prop}

\begin{proof}
We recall the Kuznetsov formula in Theorem~\ref{thm:kuzn} and, with a view towards applying it, we first perform some transformations to the quantity of interest.
By positivity we see that the left-hand side of \eqref{eq:highmoment} is bounded by
\begin{multline*}
    \sum_{LM=Dq}\sum_{\varphi\in \Bc_\mathrm{new}(Y_0(M))}\frac{h(is_\varphi)}{L(1,\sym^2 \varphi)}\mathcal{W}_\varphi(L)  \left(\sum_{p<x, p\nmid Dq} \frac{b(p)\lambda_\varphi(p)}{p^{1/2}}\right)^{2k} \\
    +\frac{1}{4\pi}\int_\R  \frac{h(t)}{|\zeta(1+2it)|^2} \mathcal{W}_t(Dq)\left(\sum_{p<x, p\nmid Dq} \frac{b(p)\lambda_t(p)}{p^{1/2}}\right)^{2k} dt,    
\end{multline*}
using here that all the terms are real numbers, see equations (\ref{eq:reallambda}), (\ref{eq:positivesym2}), (\ref{eq:Wvarphi}) and (\ref{eq:Wt}).
Define $\nu:\mathbb{N}\rightarrow \mathbb{N}$ as the multiplicative function such that $\nu(p^\alpha)=\alpha!$ and extend $b(p)$ to a completely multiplicative function on all integers.
Following Blomer--Brumley \cite[Lemma 10]{Blomer2024}, we open up the $2k$-power and apply the Hecke relations (\ref{eq:Heckerelation}) to get
\begin{equation*}
    \left(\sum_{p<x, p\nmid Dq} \frac{b(p)\lambda_\varphi(p)}{p^{1/2}}\right)^{2k}= \sum_{\substack{p|n\Rightarrow p<x, p\nmid Dq\\ \Omega(n)=2k}}\frac{(2k)!b(n)}{\nu(n)\sqrt{n}}\sum_{m|n} c_{m,n}\lambda_\varphi(m),
\end{equation*}
where $c_{m,n}\in \Z$ are the coefficients that come out of the Hecke relations,  
which are independent of $\varphi$.
In particular, for $n=\prod_{i=1}^r p_i^{k_i}$ we have by \eqref{item:hecke1} of Lemma~\ref{lem:Heckerelations} that $$c_{1,n}=\prod_{i=1}^r \left(\delta_{2|k_i}\frac{k_i!}{(\frac{k_i}{2})!(\frac{k_i}{2}+1)!}\right).$$
We obtain similar expression for the continuous spectrum. 

Since the coefficients $c_{m,n}$ are independent of $\varphi$, we can apply the Kuznetsov formula as in Corollary \ref{cor:kuzn} to the pair of integers $(m,1)$, with $m$ corresponding to a term $\lambda_\varphi(m)$ in the above sum (using here that $m\leq x^{2k}\leq (Dq)^2/4\pi^2$). This yields a diagonal contribution from $m=1$ of
\begin{equation*}
    Dq h^0 \sum_{\substack{p|n\Rightarrow p<x, p\nmid Dq\\ \Omega(n)=k}}\frac{(2k)!b(n^2)}{\nu(n^2)n}c_{1,n^2},
\end{equation*}
where we used that $c_{1,n} = 0$ if $n$ is not a square.
Note that for each prime power it holds that
\begin{equation*}
  \frac{c_{1,p^{2\alpha}}\nu(p^\alpha)}{\nu(p^{2\alpha})} = \frac{(2\alpha)!\, \alpha!}{\alpha!\,(\alpha+1)! \,(2\alpha)!}\leq 2^{-\alpha}.  
\end{equation*}
Thus, the diagonal contribution is bounded by
\begin{equation*}
    Dq h^0\frac{(2k)!}{2^k k!}\left(\sum_{p<x,p\nmid Dq}\frac{b(p)^2}{p}\right)^k,
\end{equation*}
which is admissible by Lemma~\ref{lem:h0}. The off-diagonal contribution is
\begin{equation*}
    \ll_{f_1,f_2} q^{-1/2}\sum_{\substack{p|n\Rightarrow p<x, p\nmid Dq\\ \Omega(n)=2k}}\frac{(2k)! \cdot \abs{b(n)}}{\nu(n)\sqrt{n}}\sum_{m|n} c_{m,n}\sqrt{m}, 
\end{equation*}
where the implied constant depends polynomially on $f_1,f_2$. Since $\Omega(n)=2k$, we get by \eqref{item:hecke4} of Lemma~\ref{lem:Heckerelations} that $\sum_{m| n}c_{m,n}\leq 2^{2k-1}$. Using that $m\leq n$, we bound the display above crudely by  
\begin{equation}\label{eq:boundingoffdiagonal}
    \ll q^{-1/2}\sum_{\substack{p|n\Rightarrow p<x, p\nmid Dq\\ \Omega(n)=2k}}\frac{(2k)!\cdot \abs{b(n)}2^{2k}}{\nu(n)}\ll q^{-1/2}  \left(2\sum_{p<x,p\nmid Dq}|b(p)|\right)^{2k}. 
\end{equation}
\end{proof}

Given the heuristic in Section \ref{sec:heuristics}, the bound in Proposition~\ref{prop:key} should be essentially sharp. 
For technical reasons we will also need a slight variant of the above. In this case we will settle for suboptimal bounds, which will however suffice for our purposes.
\begin{prop}[High moment estimate, variant]\label{prop:keyvariant}
    Let $x\geq 2$ and let $k,d$ be positive integers such that $x^{dk}<qD/2\pi$.  Let $b(p)$ be a sequence of real numbers indexed by primes. Then it holds that
    \begin{align}
        \label{eq:highmoment-var}
        &\sum_{\varphi\in \Bc_{\mathrm{new}}(Y_0(Dq))} \frac{h(is_\varphi)}{L(1,\sym^2 \varphi)} \left(\sum_{p<x, p\nmid Dq} \frac{b(p)\lambda_\varphi(p^d)}{p^{1/2}}\right)^{2k} \nonumber \\
        &\ll_{f_1,f_2} q (d+1)^{2k}\sum_{2\ell_1+3\ell_2=2k} \frac{(2k)!}{\ell_1!\ell_2!} \left(\sum_{p<x, p\nmid Dq} \frac{b(p)^2}{p}\right)^{\ell_1}\left(\sum_{p<x, p\nmid Dq} \frac{|b(p)|^3}{p^{3/2}}\right)^{\ell_2} \nonumber \\
        &+q^{-1/2}  \left((d+1)\sum_{p<x,p\nmid Dq}|b(p)|\right)^{2k},
    \end{align}
   with polynomial dependence on $t_{f_1},t_{f_2}$.
\end{prop}
\begin{proof}
    As in the proof of Proposition~\ref{prop:key}, we expand the $2k$-power and use the Hecke relations for $n=\prod_i p_i^{\ell_i}$ to write
    $$\prod\lambda_\varphi(p_i^d)^{\ell_i}=\sum_{m|n^d} c_{m,n}^{(d)}\lambda_\varphi(m).$$
    Note that, by \eqref{item:hecke1} of Lemma~\ref{lem:Heckerelations}, we have that $c^{(d)}_{1,n}$ is concentrated on square-full numbers $n$. 
    Now we apply the Kuznetsov formula in Corollary \ref{cor:kuzn}.
    By the bound $c^{(2)}_{0,n}\leq (d+1)^{2k-1}$ in Lemma~\ref{lem:Heckerelations}, the diagonal contribution can be bounded by
    $$ h^0 q (d+1)^{2k} \sum_{\substack{p|n\Rightarrow p<x, p^2 |n,\\\Omega(n)=2k, (n,D)=1}} \frac{(2k)!}{\nu(n)} \frac{|b(n)|}{ n^{1/2}}.   $$
    Now any square-full number $n$ can be written uniquely as $n=n_1^{2\ell_1}n_2^{3\ell_2}$ with $n_2$ square-free. 
    Extracting the relevant term in the sum in \eqref{eq:highmoment-var} by multiplicativity, we are done by the elementary inequality
    $$ \frac{1}{\nu(p^{2\ell_1+3\ell_2})}\leq \frac{1}{\ell_1!\ell_2!}. $$
    We bound the off-diagonal term using estimates similar to \eqref{eq:boundingoffdiagonal}, giving us the second term in \eqref{eq:highmoment-var}. 
\end{proof}

\section{Bounds for fractional moments of $L$-functions}\label{sec:fractionalmoment} 
In this section we will prove the main fractional moment estimate Theorem~\ref{thm:mainestimate}. The main tool for achieving this is the high moments estimates of Proposition~\ref{prop:key} and Proposition~\ref{prop:keyvariant} combined with the upper bound for central values of $L$-functions under GRH as in Proposition~\ref{prop:GRHbound}.  

Note that, by the ``in stages'' argument in Proposition~\ref{prop:in-stages} and Remark \ref{rem:t1t2}, we may assume that 
\begin{equation}
 \label{eq:tFbound} t_{f_1},t_{f_2},t_F\leq q^{1/200}.   
\end{equation} 
Furthermore, we will see in \eqref{eq:large-tphi-neglig} that we may truncate the sum at $t_\varphi \leq q^{1/100}$, in which case it holds that 
\begin{equation} \label{eq:analyticconductor}
    \log \mathbf{c}(F\otimes F\otimes \varphi), \log \mathbf{c}(f_1\otimes f_2\otimes \varphi) \asymp \log q.
\end{equation}
We will work under this assumption throughout.

\subsection{Preparing for the proof}

\subsubsection{The mean and variance}To prepare for the proof, in view of the heuristics, it will be important to consider the ``mean'' of $\log \mathcal{L}(\varphi)$, which should be given by $\mu_{F,F}(x)+\mu_{f_1,f_2}(x)$ defined in (\ref{eq:defmupi_i}), as Proposition~\ref{prop:GRHbound} shows.
Accordingly, we first put
\begin{align*}
    \mathcal{L}_{F}(s)&:= \frac{L(s,\sym^4 F)}{L(s,\sym^2 F)}  ,\\
    \mathcal{L}_{F,f}(s)&:= \frac{L(s,\sym^4 F)L(s,\sym^4 f)}{L(s,\sym^2 F)L(s,\sym^2 f)}
\end{align*}
and define the \emph{mean} as 
\begin{equation}\label{eq:muq}
    \mu_q(x):=
    \begin{cases}
     -\frac{3}{2}\log\log x-\frac{1}{2}\log \mathcal{L}_{F}(1)+B (\log \mathbf{c}(f_1\otimes f_2))^{15/16},& f_1\neq f_2,\\
     -2\log\log x-\frac{1}{2}\log \mathcal{L}_{F,f}(1),& f_1=f_2=f,
    \end{cases}
\end{equation}
where $B>0$ denotes the constant in (\ref{eq:defmupi_i}). 
Next, we write
\begin{align*}
    \mathcal{L}^{(2)}_{F,f_1,f_2}(s):= &  L(s,\sym^4 F)L(s,\sym^2 F)^3L(s,\sym^2 f_1\otimes \sym^2 f_2)\\
    & \cdot L(s,\sym^2 f_1)L(s,\sym^2 f_2)L(s,\sym^2 F\otimes f_1\otimes f_2)^2L(s, f_1\otimes f_2)^2   ,\\
    \mathcal{L}^{(2)}_{F,f}(s):=& L(s,\sym^4 F)L(s,\sym^2 F)^5L(s,\sym^2 F\otimes \sym^2 f)^2\\
    & \cdot L(s,\sym^4 f)L(s,\sym^2 f)^5,
\end{align*}
and define the \emph{variance}
\begin{equation}\label{eq:varq}
     \var_q(x):=\begin{cases}
         3\log\log x+\log \mathcal{L}^{(2)}_{F,f_1,f_2}(1),& f_1\neq f_2,\\
         6\log\log x+\log \mathcal{L}^{(2)}_{F,f}(1),& f_1=f_2=f,
     \end{cases}
\end{equation}
This agrees with the variance obtained in the heuristic in Section \ref{sec:heuristics} taking into account the $L$-function at $s=1$ in (\ref{eq:varrandommodel}).  For brevity, we put $\mu_q:=\mu_q(q)$ and $\var_q:=\var_q(q)$. We record the following basic estimates.

\begin{lemma}\label{lem:boundmuvar}
Assume GRH and that $t_{f_1},t_{f_2},t_F \leq q^{1/100}$.  Then it holds that
\begin{equation*}\log\log q\ll \var_q\ll (\log q)^{15/16},\end{equation*}
and
\begin{equation*}\mu_q\ll (\log q)^{15/16}.\end{equation*}
\end{lemma}
\begin{proof}
    Assume first that $f_1\neq f_2$. By the Hecke relations in Lemma~\ref{lem:variance-coeffs}, we see that the Dirichlet coefficient at a prime $p$ of $\mathcal{L}^{(2)}_{F,f_1,f_2}(s)$ equals
    \begin{displaymath}
        (\lambda_F(p)^2+\lambda_{f_1}(p)\lambda_{f_2}(p))^2-3 \geq -3.
    \end{displaymath}
    Thus by Lemma~\ref{lem:L(1)=sum_p}, as when proving Corollary \ref{cor:boundsL(1)}, and by (\ref{eq:analyticconductor}), we conclude that
    $$ \var_q \geq 3\log \log q - 3 \log \log \log q + O(1),$$
    The case for $f_1=f_2$ is similar and this finishes the proof for the lower bound in the first inequality.
    
    The upper bounds for $\mu_q$ and $\var_q$ follow by applying Corollary \ref{cor:boundsats=1} to bound the $L$-functions at $s=1$, with $\delta= \frac{7}{16}$ using the Kim--Sarnak bound \eqref{eq:KimSarnak} as well as (\ref{eq:analyticconductor}).
    More precisely, for $\pi \in \{ \sym^4 F, \sym^2 F, \ldots\}$, we have $\lambda_\pi(p) \ll p^{4\cdot 7/64}$, which gives a bound of $(\log q)^{2 \cdot 7/16 + \varepsilon}$ and we take $\varepsilon < 1/16$.
\end{proof}

\begin{lemma} \label{lem:fraction-L-fn-with-delta}
    Let $0 < \delta< 1$ and assume GRH. Then it holds that
    \begin{equation*}
        \frac{(\mathcal{L}^{(2)}_{F,f_1,f_2}(1))^\delta}{\mathcal{L}_{F}(1)}\ll_{f_1,f_2}(\log\log q)^{O_\delta(1)},
    \end{equation*}
   and
    \begin{equation*}
        \frac{(\mathcal{L}^{(2)}_{F,f}(1))^\delta}{\mathcal{L}_{F,f}(1)}\ll_{f_1,f_2}(\log\log q)^{O_\delta(1)}.
    \end{equation*}
    The implied constants depend polynomially on  $t_{f_1},t_{f_2}$. 
\end{lemma}
\begin{proof}
Note that the $p$-th Dirichlet coefficient of $\mathcal{L}^{(2)}_{F,f_1,f_2}(1)^\delta/\mathcal{L}_{F}(1)$ is
\begin{displaymath}
    \delta(\lambda_F(p)^2 + \lambda_{f_1}(p) \lambda_{f_2}(p))^2 - 3 \delta - \lambda_F(p)^4 + 4 \lambda_F(p)^2 - 2.
\end{displaymath}
By completing the square and removing the negative quantity $-3\delta - 2$, this is less than or equal to
\begin{displaymath}
    (\delta^2 - 1)\left( \lambda_F(p)^2 + \frac{\delta \lambda_{f_1}(p) \lambda_{f_2}(p) + 2}{(\delta^2 - 1)} \right)^2 - \frac{(\delta \lambda_{f_1}(p) \lambda_{f_2}(p) + 2)^2}{\delta^2 - 1} + (\lambda_{f_1}(p) \lambda_{f_2}(p))^2.
\end{displaymath}
Since $0 < \delta < 1$, the first term can be bounded above simply by $0$.
The other two terms together, by expanding the square and applying Hecke relations, give
\begin{multline*}
    (\delta^2-1)^{-1} (- (\lambda_{f_1}(p) \lambda_{f_2}(p))^2 + 4 \delta \lambda_{f_1}(p) \lambda_{f_2}(p) + 4) \\
    = (\delta^2-1)^{-1}(-\lambda_{f_1}(p^2) \lambda_{f_2}(p^2) - \lambda_{f_1}(p^2) - \lambda_{f_2}(p^2) + 4 \delta \lambda_{f_1}(p) \lambda_{f_2}(p) + 3).
\end{multline*}
Therefore, we see that the $p$-th coefficient of
\begin{displaymath}
    \frac{\mathcal{L}^{(2)}_{F,f_1,f_2}(s)^\delta}{\mathcal{L}_{F}(s)} \cdot \left( \frac{L(s,\sym^2 f_1\otimes \sym^2 f_2) + L(s,\sym^2 f_1)L(s,\sym^2 f_2)}{L(s, f_1 \otimes f_2)^{4\delta}} \right)^{1/(\delta^2-1)}
\end{displaymath}
is bounded from above by $3/(\delta^2 - 1) \leq 0$.
Thus, by Corollary \ref{cor:boundsL(1)} we get the first inequality, using that the extra factor above is indeed polynomially bounded in $\mathbf{c}(f_i)$ at $s=1$ by Corollary \ref{cor:boundsats=1}.
The second one follows similarly by completing the square.
\end{proof}

\begin{cor}
    Theorem~\ref{thm:mainestimate} follows from showing that, for any $\varepsilon_1,\varepsilon_2>0$ sufficiently small,
    \begin{equation}
        \label{eq:whatwewant}
        \sum_{\varphi\in \mathcal{B}_{\new}(Y_0(qD))} \frac{h(is_\varphi)}{L(1,\sym^2 \varphi)} \mathcal{L}(\varphi)^{1/2} \ll^?_{\varepsilon_i,f_i} q \exp\left( \frac{1-\varepsilon_1}{2}\mu_q+\frac{1+\varepsilon_2}{8}\var_q\right),
    \end{equation}
    with polynomial dependence on $t_{f_1}, t_{f_2}$.
\end{cor}
\begin{proof}
    Since $\frac{1+\varepsilon_2}{8}<\frac{1-\varepsilon_1}{4}$ for $\varepsilon_1,\varepsilon_2$ sufficiently small, Lemma~\ref{lem:fraction-L-fn-with-delta} indeed applies.
\end{proof} 

The rest of the section is devoted to proving the estimate \eqref{eq:whatwewant}.

\subsubsection{Applying high moment estimates, heuristically} We will now define the key quantities needed for the proof of the estimate (\ref{eq:whatwewant}) as well as give a sketch of how to carry out the proof. 

For $V\in \R$, we define the cumulative counting function:
\begin{equation*}
\mathcal{I}(V):= \sum_{\varphi\in \Bc_\new(Y_0(qD)) \colon \mathcal{L}(\varphi)>e^V} \frac{h(is_\varphi)}{L(1,\sym^2 \varphi)}. \end{equation*}
By partial summation and a change of variable we have for any $\varepsilon> 0$
\begin{align}
    \nonumber \sum_{\varphi\in \Bc_\new(Y_0(qD))} \frac{h(is_\varphi)}{L(1,\sym^2 \varphi)} \mathcal{L}(\varphi)^{1/2} & = \frac{1}{2}\int_\R \mathcal{I}(V) e^{V/2} \, dV \\
    \label{eq:integral}&= \frac{e^{(1-\varepsilon)\mu_q/2}}{2}\int_\R \mathcal{I}(V+(1-\varepsilon)\mu_q) e^{V/2} \, dV.
\end{align}
The idea is to bound $\mathcal{I}(V+(1-\varepsilon)\mu_q)$ using Propositions \ref{prop:key} and \ref{prop:keyvariant} from the previous section. Note that the extra factor $(1-\varepsilon)$ is necessary in our setting for technical reasons and does not show up in \cite{Blomer2024}, \cite{lester-radz}.  
We will now sketch the general idea, ignoring among other things the factor $(1-\varepsilon)$. 

First of all, the Kuznetsov formula shows that the contribution from $V\leq 0$ is $O(qe^{\mu_q/2})$ which is admissible (we shall prove this in Lemma~\ref{lem:smallrange}).
Next, the basic strategy is to bound the cumulative counting function for~$V>0$ using that, for any~$k \geq 1$, we have the following Markov-type inequality:
\begin{equation} \label{eq:markov-chebyshev-ineq}
    \mathcal{I}(V+\mu_q)\leq  V^{-2k} \sum_{\varphi\in \Bc_\new(Y_0(Dq))} \frac{h(is_\varphi)}{L(1,\sym^2 \varphi)} \left(\log \mathcal{L}(\varphi)-\mu_q\right)^{2k}.
\end{equation}
We are using here that the weights $\tfrac{h(is_\varphi)}{L(1,\sym^2 \varphi)}$ are positive by (\ref{eq:positivesym2}) and (\ref{eq:h(t)}), even for non-tempered $\varphi$, which should not exist according to Selberg's conjecture.
By an overly simplified version of Proposition~\ref{prop:GRHbound}, ignoring the higher order terms and the errors, forgetting the key condition $x^k \ll q$ and taking $\log\log x \sim \log\log q$, we have
\begin{displaymath}
    \left(\log \mathcal{L}(\varphi)-\mu_q\right)^{2k} \ll \left( \sum_{p < x} \frac{\lambda_{f_1}(p) \lambda_{f_2}(p) + \lambda_F(p)^2}{p^{1/2}} \lambda_\phi(p) \right)^{2k}.
\end{displaymath}
Applying now Proposition~\ref{prop:key}, where we ignore the second term in the upper bound, we would deduce that
\begin{equation*}
    \mathcal{I}(V+\mu_q)\ll q \frac{(2k)!}{k!}V^{-2k}\left(\frac{1}{2} \var_q(x) \right)^{k}\ll q \left(\frac{2\var_q(x)k}{e V^2}  \right)^{k},
\end{equation*}
using the following elementary but key inequality 
\begin{equation}\label{eq:2k/k}
    \frac{(2k)!}{k!}\leq \sqrt{2}\left(\frac{4k}{e}\right)^k.
\end{equation}
We ignored here another important condition, namely $x \geq \log (\mathbf{c}(F\otimes F\otimes \varphi))^{2+\varepsilon}$ and similarly for $f_1,f_2$, when applying Lemma~\ref{lem:L(1)=sum_p} to rewrite the short sums over Hecke eigenvalues in terms of $L$-functions to obtain $\var_q(x)$.

It is an elementary fact that $(ka/e)^k$ with $a>0$ fixed is minimized at $k=a^{-1}$. Thus we would like to pick
\begin{equation} \label{eq:k}
      k=\frac{V^2}{2\var_q(x)}
\end{equation}
In this case we would arrive at the bound
\begin{equation}\label{eq:integralheurostic}
    \int_\R \mathcal{I}(V) e^{V/2} \, dV\ll e^{\mu_q/2} \int_\R e^{-\frac{V^2}{2\var_q}+V/2} \, dV= e^{\mu_q/2} \cdot e^{\var_q/8} \sqrt{2\pi \var_q},
\end{equation}
using the formula 
\begin{equation}\label{eq:CLTformula}
    \int_{\R} e^{-\alpha x^2+\beta x} \,dx = e^{\beta^2/4\alpha}\sqrt{\frac{\pi}{\alpha}}, \quad \alpha>0, \beta\in \R.
\end{equation} 
This recovers the heuristic up to a factor of size $( \log q)^{o(1)}$.

In practice due to the restriction $x^k\ll q$, we will have to split  $\mathcal{I}(V+(1-\varepsilon)\mu_q)$ into various contributions, as well as considering different ranges of $V$. 
We will motivate the choices of splittings and ranges below. 
The main contribution will be from $V \asymp \var_q$ and in this range we would like to pick $k=\frac{V^2}{2\var_q}$.
For larger $V$, where the condition above would be violated, we instead make a suboptimal choice of (a smaller) $k$, which suffices since $V$ would be away from the bulk of the integral in (\ref{eq:integralheurostic}). 

\subsection{Separating the different contributions}\label{sec:table}
Fix throughout $\varepsilon>0$ sufficiently small. All implied constants in this section will be allowed to depend on the discriminant $D_B=\disc(B)$. 

We will now bound $\mathcal{I}(V+(1-\varepsilon)\mu_q)$, separating the contribution from the different terms in Proposition~\ref{prop:GRHbound} corresponding to $p^{-\sigma}$ with $\sigma\in \{\tfrac{1}{2},1,\ldots ,5/2\}$. 
Furthermore, for the main term $\sigma=1/2$, we will also split the sum over $p$ into different ranges (to ensure that the bound $x^k \ll q$ holds). 
We will apply this splitting according to the following simple principle: if $X,X_1,\ldots, X_n$ are real valued random variables such that $X\leq \sum_{i=1}^n X_i$ with probability one, then for any real numbers satisfying $V\geq \sum_{i=1}^n V_i$, it follows by subadditivity that
\begin{equation}\label{eq:P(X>V)}
    \mathbb{P}(X>V)\leq\sum_{i=1}^n \mathbb{P}(X_i>V_i),
\end{equation}
where $\mathbb{P}(\mathcal{E})$ denotes the probability of the event $\mathcal{E}$. 
We then apply Propositions \ref{prop:key} and \ref{prop:keyvariant} to each of the contributions separately. 
We summarize our treatment in Table \ref{table:ranges}.

\begin{table}[h]
    \begin{center}
    \caption{An overview of the ranges.\label{table:ranges}}
    \begin{tabular}{ | c | c| c |c| } 
      \hline
    \diagbox{terms with $p^{-\sigma}$}{range of $V$}  &  small & bulk & large \\ 
      \hline
    $\sigma=1/2$, small $p$  & 
    &  Lemma~\ref{lem:bulkrange} & \multirow{2}{*}{Lemma~\ref{lem:largerange1/2}}  \\
      \cline{1-1}\cline{3-3}
    $\sigma=1/2$, not small $p$   &  Lemma~\ref{lem:smallrange}&  Lemma~\ref{lem:bulkrangezx}&  \\
          \cline{1-1}\cline{3-4}
     $\sigma\in \{1,3/2,\ldots, 5/2\}$   &  & \multicolumn{2}{c|}{Lemma~\ref{lem:higherorderterms}}\\
          \hline  
    \end{tabular}
    \end{center}
\end{table}

To quantify the ranges of $V$, we define the slowly growing quantity 
\begin{equation*}
    \Delta:= \log \log \log q.
\end{equation*}
In fact, any choice such that $\Delta \leq (\log \log q)^{\delta}$, for some small enough $\delta$, and $\Delta\rightarrow \infty$ would work.

First of all, we bound the contribution from small values of $V$.
\begin{lemma}[Small range]\label{lem:smallrange}
    It holds that
    \begin{align*}
        e^{(1-\varepsilon)\mu_q/2}\int_{-\infty}^{\Delta^{-1}\var_q} \mathcal{I}(V+(1-\varepsilon)\mu_q) e^{V/2}dV\ll q e^{(1-\varepsilon)\mu_q/2+\var_q/2\Delta}. 
    \end{align*}
\end{lemma}
\begin{proof}
By Corollary \ref{cor:kuzn} and positivity we have
\begin{equation*}
 \mathcal{I}(-\infty)=\sum_{\varphi\in \Bc_\new(Y_0(qD))} \frac{h(is_\varphi)}{L(1,\sym^2 \varphi)}\ll q.
\end{equation*}
Thus, for $V<\Delta^{-1}\var_q$ we have the following trivial estimate
\begin{align*}
  e^{(1-\varepsilon)\mu_q/2}\int_{-\infty}^{\Delta^{-1}\var_q} \mathcal{I}(V+(1-\varepsilon)\mu_q) e^{V/2}dV&\ll q e^{(1-\varepsilon)\mu_q/2}  \int_{-\infty}^{\Delta^{-1}\var_q} e^{V/2}dV\\
  &\ll q e^{(1-\varepsilon)\mu_q/2+\var_q/(2\Delta)}, 
\end{align*}
as wanted. 
\end{proof}
We note also that the contribution from $t_\varphi$ very large (so that in particular $t_\varphi=is_\varphi$) is negligible in the range of spectral parameters (\ref{eq:tFbound}):
\begin{equation} \label{eq:large-tphi-neglig}
    \mathcal{I}_\mathrm{tail}:=  \sum_{\varphi\in \Bc_\new(Y_0(qD)) \colon t_\varphi> q^{1/100}} \frac{h(is_\varphi)}{L(1,\sym^2 \varphi)}\ll q^{-100},
\end{equation}
using the rapid decay of $h$,  Corollary \ref{cor:sym2at1} and the Weyl law.

Recall that, under GRH and assuming $t_{f_1},t_{f_2}\leq q^{1/100}$ and thus using (\ref{eq:analyticconductor}), we have (see \cite[Rem.~1]{Chandee09})
\begin{displaymath}
    \log \mathcal{L}(\varphi) \ll \frac{\log q}{\log \log q}.
\end{displaymath}
Using Lemma~\ref{lem:boundmuvar} to bound $\mu_q$, we deduce that there exists $A > 0$, depending on the discriminant $D$, such that
\begin{equation}\label{eq:GRHbound}
    \log \mathcal{L}(\varphi)-(1-\varepsilon)\mu_q \leq A\frac{\log q}{\log\log q},
\end{equation}
for all $\varphi\in \Bc_\new(Dq)$ with $t_\varphi \leq q^{1/100}$, using again (\ref{eq:analyticconductor}). 
Thus, $\mathcal{I}(V + (1-\varepsilon)\mu_q) = 0$ for $V \geq A\log q/\log\log q$ and, together with Lemma~\ref{lem:smallrange}, this shows that we may focus on the range 
\begin{equation}\label{eq:rangeV}
(\log\log q)^{1+o(1)} \leq \Delta^{-1}\var_q<V<A\log q/\log\log q
\end{equation}
in the integral (\ref{eq:integral}).

Throughout this section we put
\begin{equation}\label{eq:x}
    x=\exp\left( \frac{9 A \log q}{ \varepsilon V}\right),
\end{equation}
for which we will  be applying Proposition~\ref{prop:GRHbound}. 
To motivate this choice, we note that this is the minimal choice of $x$ so that the first error-term in (\ref{eq:GRHbound})
\begin{equation}\label{eq:logc}\log \mathbf{c}(F\otimes F\otimes \varphi)/\log x\ll  \log q/\log x\ll \varepsilon V,\end{equation} is small (relative to $V$).
We note that, for $\varepsilon<1/100 $ and $V$ in the range (\ref{eq:rangeV}), it holds that
\begin{equation}\label{eq:x>log} x\geq (\log q)^{9/\varepsilon}\geq (\log q)^{900},\end{equation}
so that the condition (\ref{eq:lowerboundx}) in Proposition~\ref{prop:GRHbound} is indeed satisfied, and 
\begin{equation}\label{eq:muq(x)}
|\mu_q(x)-\mu_q|\ll |\log \log x-\log\log q| \ll \log (\varepsilon V) = O(\varepsilon V), 
\end{equation}
for $q$ sufficiently large (in terms of $\varepsilon$), implying that $V$ is large using the lower bound in \eqref{eq:rangeV}.
Now for $1 \leq n \leq 5$, put 
$$\mathcal{S}_{n/2}(\varphi):=\mathcal{S}_{n/2}(F\otimes F\otimes \varphi;x)+\mathcal{S}_{n/2}(f_1\otimes f_2\otimes \varphi;x),$$
with $x$ as in (\ref{eq:x}).
Given \eqref{eq:x>log}, Proposition~\ref{prop:GRHbound} is applicable.
Under the assumption that $t_\varphi\leq q^{1/100}$,
we have 
\begin{equation}\label{eq:logL-mu}
    \log \mathcal{L}(\varphi)-(1-\varepsilon)\mu_q\leq \sum_{n=1}^5 \mathcal{S}_{n/2}(\varphi) +O(\varepsilon V+\log\log \log q), 
\end{equation} 
using the estimates (\ref{eq:logc}) and (\ref{eq:muq(x)}).

We now get ready to apply the principle given in \eqref{eq:P(X>V)}.
For this, we take the probability measure proportional to the measure given on functions $f(\varphi)$ by
\begin{displaymath}
    \sum_{\varphi\in \Bc_\new(Y_0(qD)) \colon t_\varphi < q^{1/100}} \frac{h(is_\varphi)}{L(1,\sym^2 \varphi)} f(\varphi).
\end{displaymath}
Put $X(\varphi) = \log \mathcal{L}(\varphi)-(1-\varepsilon)\mu_q$ and $X_n(\varphi) = S_{n/2}(\varphi)$ for $n = 1, \ldots, 5$.
Let also $X_{\mathrm{error}}(\varphi) = C'\varepsilon V$ for a constant $C' > 0$ such that, for $q$ large enough in terms of $\varepsilon$ and for $V$ in the range \eqref{eq:rangeV}, the inequality \eqref{eq:logL-mu} reads
\begin{displaymath}
    X \leq \sum_n X_n + X_{\mathrm{error}}.
\end{displaymath}
Setting $V_2 = \ldots = V_5 = \varepsilon V$, $V_{\mathrm{error}} = C'\varepsilon V$, and $V_1 = (1-C\varepsilon) V$ with $C = C' + 5$, we apply \eqref{eq:P(X>V)}.
Notice that $\mathbb{P}(X_{\mathrm{error}} > V_{\mathrm{error}}) = 0$.

On the right-hand side of the inequality \eqref{eq:P(X>V)}, we may extend the measure to include also the average over $t_\varphi \geq q^{1/100}$ by positivity.
Thus, for $1\leq n\leq 5$, we put 
\begin{equation*}
    \mathcal{I}_{n/2}(V):= \sum_{\varphi\in \Bc_\new(Y_0(qD)) \colon \mathcal{S}_{n/2}(\varphi)>V} \frac{h(is_\varphi)}{L(1,\sym^2 \varphi)}. 
\end{equation*}
We include the range $t_\varphi \geq q^{1/100}$ also on the left-hand side simply by adding $\mathcal{I}_\mathrm{tail}$ to the inequality.
We finally obtain the bound
\begin{equation}\label{eq:minkowskiineq1}
    \mathcal{I}(V + (1-\varepsilon)\mu_q)\leq  \mathcal{I}_{1/2}((1-C\varepsilon)V)+ \sum_{n=2}^5 \mathcal{I}_{n/2}(\varepsilon V)+\mathcal{I}_\mathrm{tail}.
\end{equation}

We will now apply Propositions \ref{prop:key} and \ref{prop:keyvariant} to bound $\mathcal{I}_{1/2}(V)$ and $\mathcal{I}_{n/2}(V)$ with $n\geq 2$, respectively. 
The former, $\mathcal{I}_{1/2}(V)$ requires more subtle arguing. 
For $V\asymp \var_q$, $k$ as in (\ref{eq:k}) and $x$ as in (\ref{eq:x}) we see that the inequality $x^k<qD/2\pi$ is violated. To avoid this we put
\begin{equation} \label{eq:z}
    z=x^{1/\Delta^2}=\exp\left(\frac{9A\log q}{\varepsilon \Delta^2 V}\right),
\end{equation}
with $\Delta=\log\log\log q$ as above, and write
\begin{equation*}
  \mathcal{S}_{1/2}(\varphi)=\sum_{p<z,p\nmid D}+\sum_{z\leq p<x,p\nmid D}=: \mathcal{S}_{1/2}(\varphi;z)+\mathcal{S}_{1/2}(\varphi; z,x), 
\end{equation*}
and put 
\begin{equation*}
 \mathcal{I}_{1/2}(V;z):= \sum_{\varphi\in \Bc_\new(Y_0(qD)) \colon \mathcal{S}_{1/2}(\varphi;z)>V} \frac{h(is_\varphi)}{L(1,\sym^2 \varphi)}, 
\end{equation*}
and similarly for $\mathcal{I}_{1/2}(V;z,x)$. We note that 
\begin{equation}\label{eq:minkowskiineq2}
    \mathcal{I}_{1/2}((1-C \varepsilon)V)\leq \mathcal{I}_{1/2}((1-(C+1)\varepsilon)V;z)+\mathcal{I}_{1/2}(\varepsilon V;z,x),
\end{equation}
and we will now apply Proposition~\ref{prop:key} to each of the two terms separately.

\subsection{Applying high moment estimates, rigorously}
In this section we will apply Propositions \ref{prop:key} and \ref{prop:keyvariant} to each choice of a pair consisting of a range for $V$ and one of the terms in Proposition~\ref{prop:GRHbound}, as in explained in Table \ref{table:ranges} in the previous section.

\subsubsection{Contribution from $p^{-1/2}$} 
Firstly, we bound the main contribution which will come from $p<z$, i.e. the first term in the right-hand side of \eqref{eq:minkowskiineq2}.
In doing so, we split into two ranges according to whether $V$ is in the bulk of the integral (\ref{eq:integralheurostic}) or not. 
Consider first the contribution from the bulk.
\begin{lemma}[Bulk range, $p<z$]\label{lem:bulkrange}
    Let $\Delta^{-1}\var_q<V\leq \Delta\var_q$. Then, for $q$ sufficiently large, it holds that
    \begin{equation*}
        \mathcal{I}_{1/2}((1-(C+1)\varepsilon)V;z)\ll q \exp\left( -(1 - (2C + 3)\varepsilon)V^2/2\var_q \right). 
    \end{equation*}
\end{lemma}
\begin{proof}
    Note that, by an inequality analogous to \eqref{eq:markov-chebyshev-ineq}, we have
    \begin{multline*}
        \mathcal{I}_{1/2}((1-(C+1)\varepsilon)V; z) 
        \\
        \leq  ((1-(C+1)\varepsilon)V)^{-2k} \sum_{\varphi\in \Bc_\new(Y_0(Dq))} \frac{h(is_\varphi)}{L(1,\sym^2 \varphi)} \left( \mathcal{S}_{1/2}(\varphi; z)\right)^{2k}.  
    \end{multline*}
    We now apply Proposition~\ref{prop:key} with $x=z$ as in equation (\ref{eq:z}) and a $k$ such that $z^k<qD/2\pi$, which we determine below.
    This yields an upper bound of
    \begin{equation}\label{eq:applyingkey}
        q\frac{(2k)!}{k!} \left( \frac{1}{2(1-(C+1)\varepsilon)^2 V^2 } \sum_{p<z,p\nmid D} \frac{b(p)^2}{p}\right)^k+q^{-1/2}  \left(2\sum_{p<z,p\nmid D}|b(p)|\right)^{2k},
    \end{equation}
    where
    \begin{equation} \label{eq:b(p)} 
        b(p)= p^{-1/\log x}|\lambda_{f_1}(p)\lambda_{f_2}(p)+\lambda_{F}(p)^2|\frac{\log (x/p)}{\log x}\leq |\lambda_{f_1}(p)\lambda_{f_2}(p)+\lambda_{F}(p)^2|.
    \end{equation}
    We simplified the second term using that $(1 - (C+1)\varepsilon)V \gg 1$.
    
    By expanding the square and applying the Hecke relations as in Lemma \ref{lem:variance-coeffs}, we see that
    \begin{align}
        b(p)^2&\leq \lambda_{f_1}(p)^2\lambda_{f_2}(p)^2+\lambda_{F}(p)^4+2\lambda_{f_1}(p)\lambda_{f_2}(p)\lambda_{F}(p)^2 \nonumber \\
        &=  3+\lambda_{\sym^2 f_1\otimes \sym^2 f_2}(p)+ \lambda_{\sym^2 f_1}(p) +\lambda_{ \sym^2 f_2}(p) +2\lambda_{ f_1\otimes f_2}(p) \nonumber \\
        &+2\lambda_{\sym^2 F\otimes f_1\otimes f_2}(p)+\lambda_{\sym^4 F}(p)+3\lambda_{\sym^2 F}(p), \label{eq:boundb(p)}
    \end{align}
    for $f_1\neq f_2$ and similarly for $f_1=f_2$. 
    In both cases we conclude that 
    \begin{equation*}
       \sum_{p<z,p\nmid D} \frac{b(p)^2}{p}\leq \var_q+O(1)\leq (1+\varepsilon)\var_q, 
    \end{equation*}
    for $q$ sufficiently large, by extending the range to $p < q$ using positivity and applying Lemma~\ref{lem:L(1)=sum_p}. 
    To bound the second term in (\ref{eq:applyingkey}), we use that
    $$|b(p)|\ll p^{7/32} ,$$
    by the Kim--Sarnak bound \eqref{eq:KimSarnak} and so 
    \begin{displaymath}
      \left(2\sum_{p<z,p\nmid D}|b(p)|\right)^{2k} \leq  \left(2 z^{1 + 7/32}\right)^{2k} \leq z^{O(k)}.
    \end{displaymath}
    Using the inequality (\ref{eq:2k/k}) we arrive at the upper bound
    \begin{equation}\label{eq:q()^k}
        \mathcal{I}_{1/2}((1-(C+1)\varepsilon)V; z) \ll q\left( \frac{2k(1+\varepsilon)\var_q}{e(1-(C+1)\varepsilon)^2 V^2 } \right)^k+q^{-1/2}z^{O(k)} .  
    \end{equation}
    We pick
    $$k=\left\lfloor \frac{(1-(C+1)\varepsilon)^2 }{1+\varepsilon} \cdot \frac{ V^2}{2 \var_q}\right\rfloor=(1+o(1))\frac{(1-(C+1)\varepsilon)^2 }{1+\varepsilon} \cdot \frac{ V^2}{2 \var_q},$$
    which is permissible since
    \begin{align} \label{eq:admissiblez}
        k\cdot \log z\leq \frac{V^2}{2\varepsilon\var_q} \frac{\log q}{\Delta^2 V}\leq \log q \frac{\Delta \var_q}{2\varepsilon \Delta^2\var_q }=o(\log q).
    \end{align}
    Note that the second term in (\ref{eq:q()^k}) is bounded by 
    $\exp(O( k\log z ))=q^{o(1)}$.
    This yields the result since 
    \begin{displaymath}
        \frac{(1-(C+1)\varepsilon)^2 }{1+\varepsilon}>1-(2C+3)\varepsilon.
    \end{displaymath}
\end{proof}

\begin{lemma}[Bulk range, $z\leq p<x$]\label{lem:bulkrangezx}
    Let $\Delta^{-1}\var_q<V<\Delta\var_q$. Then, for sufficiently large $q$, there exists a constant $c(\varepsilon)>0$ such that 
    \begin{equation*}
        \mathcal{I}_{1/2}(\varepsilon V;z,x)\ll q \exp(-c(\varepsilon) \Delta V).
    \end{equation*}
\end{lemma}
\begin{proof}
    Applying Proposition~\ref{prop:key}, we arrive at the upper bound
    \begin{align*}
        \ll q\frac{(2k)!}{k!}\left( \frac{1}{2\varepsilon^2 V^2 } \sum_{z\leq p<x,p\nmid D} \frac{b(p)^2}{p}\right)^k+q^{-1/2}  \left(2\sum_{z\leq p<x,p\nmid D}|b(p)|\right)^{2k},
    \end{align*}
    with $b(p)$ as in (\ref{eq:b(p)}). 
    Now, for $\pi \in \{\sym ^2 F, \sym^2 F\otimes f_1\otimes f_2,\ldots, \}$, one of the forms appearing in (\ref{eq:boundb(p)}), we have $|\lambda_\pi(p)|\ll p^{7/16}$ by the Kim--Sarnak bound \eqref{eq:KimSarnak}. Note also that, by Lemma~\ref{lem:boundmuvar}, in the range $\Delta^{-1}\var_q<V<\Delta\var_q$ it holds for sufficiently large $q$ that
    \begin{displaymath}
        x\geq z\geq \exp\left( O_\varepsilon\left(\frac{\log q}{\Delta^3 (\log q)^{15/16} }\right) \right)\geq (\log q)^{100}.
    \end{displaymath}
    This means that Lemma~\ref{lem:L(1)=sum_p} applies and yields
    \begin{align*}
    \sum_{z\leq p<x, p\nmid D}\frac{\lambda_\pi(p)}{p}=\sum_{ p<x, p\nmid D}\frac{\lambda_\pi(p)}{p}-\sum_{p<z, p\nmid D}\frac{\lambda_\pi(p)}{p}=O(1).
    \end{align*}
    Thus, we conclude that
    \begin{equation*}
    \sum_{z\leq p<x, p\nmid D}\frac{b(p)^2}{p}  \ll \log\log x-\log\log z= 2\log\Delta.
    \end{equation*}
    Applying the bound \eqref{eq:2k/k}, we arrive at 
    \begin{align*}
    \mathcal{I}_{1/2}(\varepsilon V;z,x) \ll 
    q\left( \frac{c'(\varepsilon) k \log \Delta  }{ V^2 } \right)^k+q^{-1/2}  \left(c' x^{23/16}\right)^{2k},
    \end{align*}
    for some constants $c'(\varepsilon),c'>0$. 
    Now, we put $k= \lfloor \varepsilon V/900 A \rfloor$, which is admissible since
    \begin{align}\label{eq:admissiblez}
    k\cdot \log x\leq \frac{\varepsilon V }{100} \frac{\log q}{\varepsilon V}\leq \frac{\log q}{100}.
    \end{align}
    The same bound shows that 
    \begin{align*}
    \left(c' x^{23/16}\right)^{2k}\ll q^{23/800+o(1)}, 
    \end{align*}
    so we arrive at the bound 
    \begin{align*}
    \mathcal{I}_{1/2}(\varepsilon V;z,x) \ll q\exp\left(-c(\varepsilon) V \log \frac{V}{\log \Delta}\right)+1.  
    \end{align*}
    Since $V/\log \Delta\geq (\log\log q)^{1+o(1)}$, we obtain the result (recalling that $\Delta=\log \log \log q$).
\end{proof}

For $V$ very large, we have the following crude estimate for the contribution from $\sigma=1/2$.

\begin{lemma}[Large range, $\sigma=1/2$]\label{lem:largerange1/2}
    Let $\Delta\var_q<V<A\log q/\log\log q$. Then there exists a constant $c(\varepsilon)>0$ such that 
    \begin{equation*}
        \mathcal{I}_{1/2}((1-(C+1)\varepsilon)V)\ll q \exp\left( -c(\varepsilon)V\log \Delta \right). 
    \end{equation*}
\end{lemma}
\begin{proof}
    We may follow the first steps of the proof of Lemma~\ref{lem:bulkrangezx} and arrive at the analogue of \eqref{eq:q()^k}, namely
    \begin{equation*}
        \mathcal{I}_{1/2}((1-(C+1)\varepsilon)V) \ll q \left( \frac{2k(1+\varepsilon)\var_q}{e(1-(C+1)\varepsilon)^2 V^2 } \right)^k+q^{-1/2} x^{O(k)} .  
    \end{equation*}
    We pick $k=\lfloor \varepsilon V/900 A\rfloor $, which is admissible since
    \begin{align}\label{eq:admissible}
        k\cdot \log x\leq \frac{\varepsilon V}{100} \frac{\log q}{\varepsilon V}\leq \frac{\log q}{100}.
    \end{align}
    For some $c'(\varepsilon)>0$, we obtain the bound
    \begin{displaymath}
        \ll q \left( c'(\varepsilon) \frac{\var_q}{V} \right)^{\varepsilon V},
    \end{displaymath}
    which implies the result since $\var_q/V\leq \Delta^{-1}$.
\end{proof}

\subsubsection{Contribution from $p^{-n/2}$ with $n\geq 2$}
Next, we bound the contribution with $n\in \{2,\ldots, 5\}$, which can be seen to be negligible even using crude estimates.

\begin{lemma}[Higher order terms]\label{lem:higherorderterms}
    Let $\Delta^{-1}\var_q<V<A \log q/\log\log q$. Then, for $q$ sufficiently large, there exists a constant $c(\varepsilon)>0$ such that for $2\leq n\leq 5$:
    \begin{equation*}
        \mathcal{I}_{n/2}(\varepsilon V)\ll q \exp\left( -c(\varepsilon)V\log V \right). 
    \end{equation*}
\end{lemma}
\begin{proof}
    We recall for the convenience of the reader that
    \begin{multline*}
        \mathcal{S}_{n/2}(\pi_1\otimes \pi_2\otimes \pi_3;x )\\
        :=\frac{1}{n}\sum_{p^n \leq x, p\nmid D}\frac{\prod_{i=1}^2 (\sum_{m=0}^nc(m,n)\lambda_{\pi_i}(p^m))(\sum_{m=1}^nc(m,n)\lambda_{\pi_3}(p^m))}{p^{n/2+n/\log x}}\frac{\log (x/p^n)}{\log x},
    \end{multline*}
    for certain constants $c(m,n)\in \Z$.
    We pull out the sum $\sum_{m=1}^n c(m,n)\lambda_{\varphi}(p^m)$ in $\mathcal{S}_{n/2}(\varphi)$ and thus organize the latter according to the factors $\lambda_\varphi(p^m)$ for $1 \leq m \leq n$, writing
    \begin{displaymath}
        \mathcal{S}_{n/2}(\varphi) = \sum_{m = 1}^n \mathcal{S}_{n/2, m}(\varphi).
    \end{displaymath}
    By the principle of subadditivity in \eqref{eq:P(X>V)}, we can bound $\mathcal{I}_{n/2}(\varepsilon V)$ by the sum of
    \begin{displaymath}
        \mathcal{I}_{n/2, m}(\varepsilon V / n):= \sum_{\varphi\in \Bc_\new(Y_0(qD)) \colon \mathcal{S}_{n/2, m}(\varphi) > \varepsilon V / n} \frac{h(is_\varphi)}{L(1,\sym^2 \varphi)}. 
    \end{displaymath}
    
    We follow the same argument as in the proof of Lemma~\ref{lem:bulkrangezx}, using now Proposition~\ref{prop:keyvariant}. 
    We conclude that $\mathcal{I}_{n/2}(\varepsilon V)$ is bounded by 
    \begin{multline*}
    q \frac{6^{2k}}{(\varepsilon V / n)^{2k}} \sum_{2\ell_1+3\ell_2=2k} \frac{(2k)!}{\ell_1!\ell_2!} \left(\sum_{p<\sqrt{x}, p\nmid D} \frac{b(p)^2}{p^{n}}\right)^{\ell_1}\left(\sum_{p<\sqrt{x}, p\nmid D} \frac{|b(p)|^3}{p^{3n/2}}\right)^{\ell_2} \\
    +q^{-1/2} \frac{6^{2k}}{(\varepsilon V / n)^{2k}}  \left(\sum_{p<\sqrt{x},p\nmid D}\frac{|b(p)|}{p^{n/2}}\right)^{2k},
    \end{multline*} 
    where
    \begin{align*}
    |b(p)|&=\left|\left(\sum_{m=0}^nc(m,n)\lambda_{f_1}(p^n)\right)\left(\sum_{m=0}^nc(m,n)\lambda_{f_2}(p^n)\right)+\left(\sum_{m=0}^nc(m,n)\lambda_{F}(p^n)\right)^2\right|\\
    &\ll p^{7n/32}.
    \end{align*}
    where we used the Kim--Sarnak bound \eqref{eq:KimSarnak}. 
    We see that the series over~$p$ are convergent and we can find $c > 0$ such that
    \begin{displaymath}
        \sum_{p<\sqrt{x}, p\nmid D} \frac{b(p)^2}{p^{n}} < c, \quad \sum_{p<\sqrt{x}, p\nmid D} \frac{|b(p)|^3}{p^{3n/2}} < c.
    \end{displaymath}
    Inserting this, for $q$ sufficiently large, we arrive at the upper bound
    \begin{equation*}
    q \sum_{2\ell_1+3\ell_2=2k} \frac{(2k)!}{\ell_1!\ell_2!} \left(\frac{30}{\varepsilon V}\right)^{2k} c^{\ell_1+\ell_2}+q^{-1/2} (36x)^k,  
    \end{equation*}
    for some constant $c>1$. Using that $\ell_1+\ell_2\leq k$ and the crude estimate 
    $$ \sum_{2\ell_1+3\ell_2=2k} \frac{1}{\ell_1!\ell_2!} \ll k\cdot \frac{1}{k!},$$
    we arrive at 
    \begin{align*}
    q k\left(\frac{c'k}{\varepsilon^2 V^2}\right)^{k}+q^{-1/2} (36 x)^k, 
    \end{align*}
    for some constant $c'>0$. 
    We now pick $k=\lfloor \varepsilon V/900 A \rfloor$, which is admissible since $k\log x\leq \frac{\varepsilon V}{100} \frac{\log q}{\varepsilon V}= \frac{\log q}{100}$. 
    In particular, $(cx)^k\ll q^{1/100+o(1)}$.
    The bound now follows for $q$ large enough by observing that $V$ goes to infinity when $q \to \infty$ in our range.
\end{proof}

\subsection{Putting everything together}
By inserting the bounds from Lemmas \ref{lem:smallrange}, \ref{lem:bulkrange}, \ref{lem:bulkrangezx} and   \ref{lem:higherorderterms}, which cover the entire range of $V$ by the inequalities (\ref{eq:GRHbound}), (\ref{eq:minkowskiineq1}) and (\ref{eq:minkowskiineq2}) , we have for any~$\varepsilon>0$ sufficiently small and $q$ sufficiently large in terms of $\varepsilon$ that 
\begin{align*}
     &\langle |F|^2, f_1 \cdot (l_qf_2)\rangle_q\\
     \ll &\frac{1}{q L(1,\sym^2 F)}\sum_{\varphi\in \Bc_\new(Y_0(qD))} \frac{h(is_\varphi)}{L(1,\sym^2 \varphi)} \mathcal{L}(\varphi)^{1/2}\\
     \ll& \frac{e^{(1-\varepsilon)\mu_q/2}}{q L(1,\sym^2 F)} \Biggr(\int_{\Delta^{-1} \var_q}^{A\frac{\log q}{\log\log q}} e^{V/2}\Biggr(\mathcal{I}_{1/2}((1-O(\varepsilon))V;z)+\mathcal{I}_{1/2}(\varepsilon V;z,x)\\
&+\sum_{n=2}^5\mathcal{I}_{n/2}(\varepsilon V)+\mathcal{I}_\mathrm{tail}\Biggr)dV+ O(1)\Biggr)\\
     \ll &\frac{e^{(1-\varepsilon)\mu_q/2}}{ L(1,\sym^2 F)} \Biggr(\int_{\Delta^{-1} \var_q}^{\Delta \var_q} e^{V/2-(1-O(\varepsilon))V^2/2\var_q}dV+\int_{\Delta^{-1} \var_q}^{\Delta \var_q}e^{V/2-c(\varepsilon)V\Delta}dV\\
     &+\int_{\Delta \var_q}^{A\frac{\log q}{\log\log q}} e^{V/2-c(\varepsilon)V\log V}dV+\int_{\Delta^{-1} \var_q}^{A\frac{\log q}{\log\log q}} e^{V/2-c(\varepsilon)V\log \Delta}dV +q^{-100}\Biggr).
\end{align*}
We see that the last three integrals can be bounded by $O_\varepsilon(1)$ (since the exponent is negative for $q$ sufficiently large) and so the total contribution from these terms is $O(e^{(1-\varepsilon)\mu_q/2}/L(1,\sym^2 F))$.
To estimate the remaining integral we extend the range of integration and apply the formula (\ref{eq:CLTformula}) to arrive at the bound
\begin{align*}
    \ll \frac{\sqrt{\var_q}e^{(1-\varepsilon)\mu_q/2+(1-O(\varepsilon))^{-1}\var_q/8}}{ L(1,\sym^2 F)}. 
\end{align*}
Since $L(1,\sym^2 F)\gg (\log \log q)^{-O(1)}$ from Corollary \ref{cor:sym2at1} we conclude the bound (\ref{eq:whatwewant}) and consequently Theorem~\ref{thm:mainestimate} follows.

\section*{Ackowledgements}
We are grateful to Valentin Blomer and Farrell Brumley for enlightening conversations regarding our conjecture. We would also like to thank Edgar Assing and Félicien Comtat for useful discussions. 

The first author was supported by a public grant from the Fondation Math\'{e}matique Jacques Hadamard.
The second author has received funding from the European Union’s Horizon 2020 research and innovation programme under the Marie Skłodowska-Curie grant agreement No 101034255.

\printbibliography

\end{document}